\documentclass[a4paper,reqno,11pt]{amsart}

\usepackage{amsmath, amsfonts, amssymb, amsthm, amscd, fge, nccmath}
\usepackage[a4paper,scale={0.78,0.78},marginratio={1:1},footskip=7mm,headsep=10mm]{geometry}
\usepackage{tabularx}
\usepackage{subfigure}
\usepackage{psfrag}
\usepackage{perpage}
\usepackage{url}
\usepackage{color}
\usepackage{mathrsfs}
\usepackage{dsfont} 
\usepackage[usenames,dvipsnames]{xcolor}
\usepackage[linktocpage=true]{hyperref}
\usepackage{upgreek}

\usepackage[low-sup]{subdepth}    
\usepackage{color}
\usepackage{multicol}
\usepackage{mathtools}
\usepackage{mathdots}
\usepackage{mathabx}
\usepackage{tikz,tikz-cd}
\usepackage{caption}
\usepackage{ragged2e}
\usepackage[utf8]{inputenc}
\usepackage[T1]{fontenc}
\usepackage{microtype}
\usepackage{accents}
\usepackage{stackengine}
\usepackage{hyperref}
\usepackage[vcentermath]{youngtab}
\justifying
\usepackage{comment}
\usepackage{color,soul}
\usepackage{xcolor}
\usepackage{caption}
\usepackage{bbold} 
\usepackage[mathscr]{eucal}
\usepackage{relsize}
\usepackage[colorinlistoftodos]{todonotes}
\makeatletter
\def\@secnumfont{\bfseries}

\def\section{\@startsection{section}{1}%
  \z@{.7\linespacing\@plus\linespacing}{.5\linespacing}%
  {\normalfont\large\bfseries\centering}}

\def\subsection{\@startsection{subsection}{2}%
  \z@{.5\linespacing\@plus.7\linespacing}{-.5em}%
  {\normalfont\bfseries}}

\def\subsubsection{\@startsection{subsubsection}{3}%
  \z@{.5\linespacing\@plus.7\linespacing}{-.5em}%
  {\normalfont}}

\def\specialsection{\@startsection{section}{1}%
  \z@{\linespacing\@plus\linespacing}{.5\linespacing}%
  {\normalfont\centering\large\bfseries}}
\makeatother

\makeatletter

\renewenvironment{proof}[1][\proofname]{\par
\pushQED{\qed}%
\normalfont \topsep4\p@\@plus4\p@\relax
\trivlist
\item[\hskip\labelsep
\bfseries
#1\@addpunct{.}]\ignorespaces
}{%
\popQED\endtrivlist\@endpefalse
}
\makeatother

\makeatletter
\newcommand \Dotfill {\leavevmode \leaders \hb@xt@ 6pt{\hss .\hss }\hfill \kern \z@}
\makeatother

\makeatletter
\def\@tocline#1#2#3#4#5#6#7{\relax
  \ifnum #1>\c@tocdepth 
  \else
    \par \addpenalty\@secpenalty\addvspace{#2}%
    \begingroup \hyphenpenalty\@M
    \@ifempty{#4}{%
      \@tempdima\csname r@tocindent\number#1\endcsname\relax
    }{%
      \@tempdima#4\relax
    }%
    \parindent\z@ \leftskip#3\relax \advance\leftskip\@tempdima\relax
    \rightskip\@pnumwidth plus4em \parfillskip-\@pnumwidth
    #5\leavevmode\hskip-\@tempdima
      \ifcase #1
       \or\or \hskip 1.65em \or \hskip 3.3em \else \hskip 4.95em \fi%
      #6\nobreak\relax
    \Dotfill
    \hbox to\@pnumwidth{\@tocpagenum{#7}}\par
    \nobreak
    \endgroup
  \fi}
\makeatother

\makeatletter
\def\l@section{\@tocline{1}{0pt}{1pc}{}{}}
\renewcommand{\tocsection}[3]{%
\indentlabel{\@ifnotempty{#2}{\ignorespaces#1 #2.\hskip 0.7em}}#3}
\def\l@subsection{\@tocline{2}{0pt}{1pc}{5pc}{}}

\def\l@subsubsection{\@tocline{3}{0pt}{1pc}{7pc}{}}

\makeatother

\numberwithin{equation}{section}

\newtheoremstyle{mytheorem}{.7\linespacing\@plus.3\linespacing}{.7\linespacing\@plus.3\linespacing}%
     {\itshape}
     {}
     {\bfseries}
     {. }
     {0.3ex}
     {\thmname{{\bfseries #1}}\thmnumber{ {\bfseries #2}}\thmnote{ (#3)}}  

\theoremstyle{mytheorem}

\newtheorem{theorem}{Theorem}[section]
\newtheorem{lemma}[theorem]{Lemma}
\newtheorem{proposition}[theorem]{Proposition}

\newtheorem{remark}[theorem]{Remark}

\newcommand{\bbE}{{\ensuremath{\mathbb E}} }

\newcommand{\bbP}{{\ensuremath{\mathbb P}} }

\newcommand{\cN}{{\ensuremath{\mathcal N}} }

\newcommand\sfL{\mathsf L}

\newcommand\sfP{\mathsf P}
\newcommand\sfQ{\mathsf Q}

\newcommand\sfT{\mathsf T}
\newcommand\sfU{\mathsf U}

\newcommand\sfX{\mathsf X}

\newcommand\sfa{\mathsf a}

\newcommand\sft{\mathsf t}

\newcommand\sfx{\mathsf x}

\renewcommand{\tilde}{\widetilde}          
\DeclareMathSymbol{\leqslant}{\mathalpha}{AMSa}{"36} 
\DeclareMathSymbol{\geqslant}{\mathalpha}{AMSa}{"3E} 
\DeclareMathSymbol{\eset}{\mathalpha}{AMSb}{"3F}     
\newcommand{\dd}{\text{\rm d}}             

\newcommand{\sumtwo}[2]{\sum_{\substack{#1 \\ #2}}} 
\newcommand{\sumthree}[3]{\sum_{\substack{#1 \\ #2 \\ #3}}} 

\newcommand{\R}{\mathbb{R}}

\newcommand{\Z}{\mathbb{Z}}
\newcommand{\N}{\mathbb{N}}

\newcommand{\PEfont}{\mathrm}

\newcommand{\p}{\ensuremath{\PEfont P}}
\newcommand{\e}{\ensuremath{\PEfont E}}

\newcommand{\E}{\e}
\renewcommand{\P}{\p}

\newcommand{\ind}{\mathds{1}}

\renewcommand{\epsilon}{\varepsilon}
\renewcommand{\theta}{\vartheta}
\renewcommand{\rho}{\varrho}

\newcommand{\ms}{\scriptscriptstyle}

\newenvironment{myenumerate}{%
\renewcommand{\theenumi}{\arabic{enumi}}%
\renewcommand{\labelenumi}{{\rm(\theenumi)}}%
\begin{list}{\labelenumi}
	{%
	\setlength{\itemsep}{0.4em}%
	\setlength{\topsep}{0.5em}%
	\setlength\leftmargin{2.45em}%
	\setlength\labelwidth{2.05em}%
	\setlength{\labelsep}{0.4em}%
	\usecounter{enumi}%
	}%
	}%
{\end{list}
}

{\end{list}
}

{\end{list}
}

{\end{myenumerate}}

\newenvironment{myitemize}{%
\begin{list}{$\bullet$}%
 	{%
	\setlength{\itemsep}{0.4em}%
	\setlength{\topsep}{0.5em}%
	\setlength\leftmargin{2.45em}%
	\setlength\labelwidth{2.05em}%
	\setlength{\labelsep}{0.4em}%
	}%
	}%
{\end{list}}

\renewenvironment{itemize}{
\begin{myitemize}}%
{\end{myitemize}}

\MakePerPage[2]{footnote} 

\def\dd{\mathrm{d}}

\newcommand\bx{\boldsymbol x}
\newcommand\by{\boldsymbol y}

\newcommand\bz{\boldsymbol z}

\definecolor{cadmiumgreen}{rgb}{0.0, 0.42, 0.24}
\definecolor{red(munsell)}{rgb}{0.95, 0.0, 0.24}

\renewcommand{\ind}{\mathbb{1}}

\newcommand{\norm}[1]{\left\lVert#1\right\rVert} 

\definecolor{dukeblue}{rgb}{0.0, 0.0, 0.61}

\hypersetup{
	colorlinks=true, linkcolor=dukeblue,
	citecolor=dukeblue
}

\DeclareMathOperator\supp{supp}
\newcommand{\bht}{\hat{\beta}}

\begin{document}

\title[]{Moments of the  2d directed polymer in the subcritical regime
    and a generalisation of the Erd\"os-Taylor theorem}

\author{Dimitris Lygkonis, Nikos Zygouras}
\address{Department of Mathematics,
    University of Warwick,
    Coventry CV4 7AL, UK}
\email{dimitris.lygkonis@warwick.ac.uk, n.zygouras@warwick.ac.uk}

\date{\today}

\begin{abstract}
    We compute the limit of the moments of the partition function $Z_{N}^{\beta_N}
    $ of the directed polymer in dimension $d=2$  in the subcritical regime, i.e. when the inverse temperature is
    scaled as $\beta_N \sim \hat{\beta} \sqrt{\tfrac{\pi}{\log N}}$ for $\hat{\beta} \in (0,1)$.
    In particular, we establish that for every $h \in\R$,
    $	\lim_{N \to \infty } \bbE\big[\big(Z_{N}^{\beta_N}\big)^h
            \big]=\big(\frac{1}{1-\hat{\beta}^2}\big)^{\frac{h(h-1)}{2}}.$
    We also identify the limit of the moments of the averaged field
    $\tfrac{\sqrt{\log N}}{N} \sum_{x \in \Z^2}  \varphi(\tfrac{x}{\sqrt{N}})\big(Z_{N}^{\beta_N}
        (x)-1 \big )$, for $\varphi \in C_c(\R^2)$,
    as those of a gaussian free field.

    As a byproduct, we identify the limiting probability distribution of the total pairwise collisions between $h$ independent,
    two dimensional random walks starting at the origin. In particular, we derive that
    \begin{align*}
        \frac{\pi}{\log N}\sum_{1 \leq i<j\leq  h} \sfL_N^{(i,j)}\xrightarrow[N \to \infty ]{(d)} \Gamma\big( \tfrac{h(h-1)}{2},1\big) \, ,
    \end{align*}
    where ${\sfL}^{(i,j)}_N$ denotes the collision local time by time $N$ between copies $i,j$ and $\Gamma$ denotes the Gamma distribution.
    This generalises a classical result of Erd\"{o}s-Taylor \cite{ET60}.
\end{abstract}

\maketitle
\tableofcontents
\section{Introduction and main results}
Let $S=(S_n)_{n \geq 0}$ be a two dimensional simple random walk and $\big(\omega(n,z)\big)_{(n,z) \in \N \times \Z^2}$ a space-time field of i.i.d. random variables with $\bbE[\omega]=0$, $\bbE[\omega^2]=1$ and $\lambda(\beta):=\log \bbE[e^{\beta \omega}]< \infty$ for all $\beta>0$. We use the notation $\P_{a,x}$ and $\E_{a,x}$ to denote the probability and the expectation with respect to the distribution of the random walk when the walk starts from $x \in \Z^2$ at time $a\in\N$. If either
$a$ or $x$ are zero, we will omit them from the subscripts. We consider the (point-to-point) partition function
\begin{align} \label{pf_def}
    Z_N^{\beta}(x, y)= \E_x \bigg[ e^{\sum_{n=1}^{N-1} \beta  \omega(n,S_n)-\lambda(\beta)} \, \ind_{\{S_N=y\}}\bigg] \,
\end{align}
of the directed polymer, i.e. random walk,  in the random environment $\omega$, at inverse temperature $\beta>0$. We also denote the point-to-line
partition function
\begin{align}\label{Zp2L}
    Z_N^{\beta}(x):=\sum_{y\in\Z^2} Z_N^{\beta}(x, y),
\end{align}
and simply write $Z_N^{\beta}$ if $x=0$.

It was observed in \cite{CSZ17b} that, while for any fixed $\beta>0$, $Z_{N}^{\beta} \xrightarrow[N \to \infty]{} 0, \,\,\bbP-\text{a.s.}$ \cite{CSY03, C17},
an intermediate disorder regime with a phase transition arises when one scales the inverse temperature like
\begin{align} \label{beta_scaling}
    \beta_N \sim \hat\beta \sqrt{\frac{\pi}{\log N}} \qquad \text{ with }\hat\beta >0 \, .
\end{align}
In particular, it was shown in \cite{CSZ17b}, that for $\beta_N \sim \hat\beta \sqrt{\frac{\pi}{\log N}} \text{ with } \hat\beta \in (0,1) $,
\begin{align*}
    Z_{N}^{\beta_N} \xrightarrow[N \to \infty]{(d)} \exp(\rho_{\hat\beta} \sfX- \tfrac{1}{2} \rho^2_{\hat{\beta}})\, ,
\end{align*}
where $\sfX \sim \cN(0,1)$ and $\rho^2_{\hat\beta}=\log \big(\frac{1}{1-\hat\beta^2}\big)$, while for $\hat\beta\geq 1$, $Z_{N}^{\beta_N} $ converges in distribution to $0$.

One can guess the emergence of such intermediate scaling as follows. Using gaussian environment for simplicity one has that
\begin{align} \label{scnd_mom}
    \bbE\Big[\big(Z^{\beta}_{N}\big)^2\Big]=\E^{\otimes 2} \Big [e^{\beta^2\, \sfL^{(1,2)}_N}\Big] 
    =\E \Big [e^{\beta^2\, \sfL_N}\Big] \, ,
\end{align}
where $\E^{\otimes 2}$ denotes the law of two independent, $2d$ simple random walks starting both at the origin,
$\sfL^{(1,2)}_N:=\sum_{n = 1}^{N-1} \ind_{\{S^1_n=S^2_n\}}$ denotes their \textit{collision local time}
up to time $N-1$ and $\sfL_N:=\sum_{n=1}^{N-1} \ind_{\{S_{2n}=0\}}$ denotes the number of returns to zero,
up to time $N-1$, of a single $2d$ simple random walk starting at $0$.
The second equality in \eqref{scnd_mom} follows since $\sfL^{1,2}_N \stackrel{\text{law}}{=} \sfL_N $.
A classical result of Erd\"os-Taylor \cite{ET60} states that
\begin{align}\label{thm:ET}
    \frac{\pi}{\log N}\, \sfL_N \xrightarrow[N \to \infty]{(d)} \text{Exp}(1) \, ,
\end{align}
where $\text{Exp}(1)$ denotes an exponential random variable of parameter $1$.
Thus, it is not hard to see that under \eqref{beta_scaling}, one has $\sup_{N \geq 1} \bbE\big[(Z^{\beta_N}_N)^2\big]< \infty$
if and only if $\hat\beta<1$.

It is an interesting and non-trivial question whether all moments of the point-to-plane partition function
$Z_N^{\beta_N}$ remain uniformly bounded as $N\to\infty$ in the same regime
of $\hat\beta$ where the second moment remains uniformly bounded.
Information on moments higher than two in the subcritical regime has already appeared necessary in a number of situations, in particular in
proving tightness and regularity properties of the approximations to the solutions of the $2d$-KPZ \cite{CD20} or
Edwards-Wilkinson universality for the $2d$-KPZ \cite{CSZ20, G20} and \cite{T22} for the nonlinear SHE.
The lack of control on higher moments was resulting into restrictions to strict subsets of the subcritical regime in \cite{CD20, G20}, while
this was circumvented in \cite{CSZ20} by employing hypercontractivity to show, for any $\hat\beta<1$, the uniform boundedness of moments
    {\it up to} certain order $h(\hat\beta)>2$ with $\lim_{\hat\beta\uparrow 1}h(\hat\beta) =2$.
    The case of the nonlinear SHE in the entire subcritical regime is still open \cite{T22} and the techniques here could be useful there.
    Moreover, having bounds on the moments can be useful in obtaining finer results relating to the rate of convergences \cite{DG22+}.

The first result of this work is to show that {\it all}  moments of the {\it point-to-plane} partition function $Z_N^{\beta_N}$ are uniformly
bounded in the whole subcritical regime $\hat\beta<1$ while, obviously, no moment higher than one exists in the limit at $\hat\beta\geq 1$.
Note that this is in contrast to what happens in the weak disorder regime of $d\geq 3$ polymers, where the moments are
gradually reduced to just existence of the first moment as the critical point is approached \cite{BS10, J22}.

Combining this with the distributional convergence \eqref{beta_scaling} we can actually compute the limit of all moments. In particular,
our first theorem is stated as:
\begin{theorem} \label{onept_mom}
    Consider the point-to-line partition function $Z_N^{\beta_N}$ defined in \eqref{Zp2L}
    with an intermediate disorder scaling $\beta_N$ as in \eqref{def:beta-scale}, which is asymptotically equivalent to
    \eqref{beta_scaling}.
    Then, for every $\hat\beta \in (0,1)$ and  $h\geq 0$, it holds that
    \begin{align} \label{limiting_moments}
        \lim_{N \to \infty} \bbE\Big[\big(Z_{N}^{\beta_N}\big)^h
            \Big]=\bigg( \frac{1}{1-\hat\beta^2}\bigg)^{\frac{h(h-1)}{2}}=\bigg( \lim_{N \to \infty} \bbE\Big[\big(Z_{N}^{\beta_N}\big)^2
            \Big]\bigg)^{\frac{h(h-1)}{2}} \, .
    \end{align}
    Furthermore, \eqref{limiting_moments} is valid also for all $h<0$ if we assume that the law of $\omega$ satisfies the following concentration property:
    There exists $\gamma>1$ and constants $c_1,c_2 \in (0,\infty)$ such that for all $n \in \N$, $(\omega_1,\dots,\omega_n)$ i.i.d. and all convex, $1$-Lipschitz functions $f: \R^n\to \R$,
    \begin{align} \label{conc_pty}
        \bbP\Big(\big| f(\omega_1,\dots,\omega_n)-M_f\big|\geq t\Big) \leq c_1 \, \exp\Big(-\frac{t^\gamma}{c_2}\Big) \, ,
    \end{align}
    where $M_f$ is a median of $f$.
\end{theorem}

\begin{remark}
   We note that \eqref{conc_pty} is satisfied if $\omega$ is bounded or has a density of the form $\exp(-V+U)$ for $V,U:\R \to \R$, where $V$ is strictly convex and $U$ is bounded, see \cite{Led01}.
\end{remark}

\vskip 2mm
Furthermore, we compute the asymptotics of the moments of the logarithmically scaled and averaged
field
\begin{align*}
    \frac{\sqrt{\log N}}{N} \sum_{x \in \Z^2}  \varphi(\tfrac{x}{\sqrt{N}})\big(Z_{N}^{\beta_N}
    (x)-1 \big ),
\end{align*} 
for $\varphi \in C_c(\R^2)$.
In particular, we establish that
\begin{theorem} \label{avg_field_thm}
    Let $\varphi \in C_c(\R^2)$ and consider the centred and averaged field with respect to $\varphi$, that is
    \begin{align*}
        \widebar{Z}^{\beta_N}_{N}(\varphi,1):=\frac{1}{N} \sum_{x \in \Z^2} \varphi(\tfrac{x}{\sqrt{N}})\,\big(Z^{\beta_N}_{N}(x)-1\big)\, \, .
    \end{align*}
    Then, for every  $h \in \N$ with $h\geq 2$ and $\bht
        \in (0,1)$,
    \begin{align*}
        \lim_{N \to \infty} (\log N)^{\frac{h}{2}} \,  \bbE\Big[\widebar{Z}^{\beta_N}_{N}(\varphi)^h \Big] = \begin{cases}
            \rho_{\varphi}(\bht)^{h}\cdot (h-1)!! & , \text{ if $h$ is even}     \\
            0                                       & , \text{ if $h$ is odd} \, ,
        \end{cases}
    \end{align*}
    where $\rho_{\varphi}(\bht)$ is defined as
    \begin{align*}
        \rho^2_{\varphi}(\bht):= \frac{\pi \,\bht^2}{1-\bht^2}\int_{0}^1 \dd t \int_{(\R^2)^2} \dd x \, \dd y \,\varphi(x) g_t(x-y) \varphi(y),
    \end{align*}
    with $g_t(x):=\frac{1}{2 \pi t} \, e^{-|x|^2/2t}$ the two-dimensional heat kernel.
\end{theorem}

\vskip 2mm
Theorem \ref{onept_mom} in combination with an analogous to \eqref{scnd_mom} computation for the $h$ moment will, almost immediately,
lead us to a generalisation of the Erd\"os-Taylor theorem (see \cite{ET60} and  \cite{GS09} for a quenched path generalisation). 
The generalisation amounts to 
the quantity of {\it total pairwise collision times} of $h$ (instead of just two as in \cite{ET60,GS09}) independent, two-dimensional simple 
planar random walks.
More specifically, let $\Gamma(a,1)$ denote the Gamma distribution, which is the law with density function
$\tfrac{1}{\Gamma(a)} x^{a-1}e^{-x}\, \ind_{\{x>0\}}$, where in the last expression $\Gamma(a)$ is the gamma function. Then,

\begin{theorem} \label{loc_times}
    Consider $h \in \N$ such that $h \geq 2$ and  for $i=1,\dots,h$ let $S^{(i)}=\big(S^{(i)}_n\big)_{n \geq 0}$ be independent simple random walks in $\Z^2$ starting all from the origin at time zero. Moreover, for $1\leq i <j \leq h$ let
    \begin{align*}
        \sfL^{(i,j)}_N:=\sum_{n = 1}^N \ind_{\{S^{(i)}_n=S_n^{(j)}\}},
    \end{align*}
    denote the collision local time of $S^{(i)}$ and $S^{(j)}$ until time $N$. Then
    \begin{align*}
        \frac{\pi}{\log N} \sum_{1 \leq i<j \leq h} \sfL^{(i,j)}_N \xrightarrow[N \to \infty]{(d)}\Gamma\big( \tfrac{h(h-1)}{2},1\big)  \, .
    \end{align*}
    More precisely, if $Y_N:=\frac{\pi}{\log N} \sum_{1 \leq i<j \leq h} \sfL^{(i,j)}_N$, $Y$ is a random variable with law $\Gamma\big( \tfrac{h(h-1)}{2},1\big)$ and $M_{Y_N}(t)$, $M_{Y}(t)$ denote the associated moment generating functions, respectively, we have that
    \begin{align*}
        M_{Y_N}(t) \xrightarrow[N \to \infty]{} M_{Y}(t) \, ,
    \end{align*}
    for all $t \in (0,1):=I$, which is the maximum interval $I \subset (0,\infty)$ where $M_Y(t)<\infty, \, t \in I$.
\end{theorem}

\vskip 2mm
The main step towards the above theorems is to establish that, in the subcritical regime,
the moments of the two-dimensional {\it point-to-line} partition function $Z_N^{\beta_N}$
are uniformly bounded.  To state the corresponding theorem, let us briefly introduce the
    {\it averaged partition functions}.
For test functions  $\varphi,\psi: \R^2 \to \R$ such that $\varphi$ has compact support and $\psi$ is bounded, we define the averaged partition function
to be
\begin{align} \label{def:avg-field}
    Z^{\beta_N}_{N}(\varphi,\psi):=\frac{1}{N} \sum_{x,y} \varphi(\tfrac{x}{\sqrt{N}})\,Z^{\beta_N}_{N}(x,y)\, \psi(\tfrac{y}{\sqrt{N}}) \, .
\end{align}
Moreover, introduce its centred version as
$\bar Z^{\beta_N}_{N}(\varphi,\psi ):=Z^{\beta_N}_{N}(\varphi,\psi) - \bbE\big[Z^{\beta_N}_{N}(\varphi,\psi)\big]$ and, similarly, 
the centred version of the point-to-line partition function \eqref{Zp2L} as $\bar Z^{\beta_N}_{N}:=Z^{\beta_N}_{N} - \bbE[Z^{\beta_N}_{N}]$.

The key estimate of this paper is the following:
\begin{theorem} \label{mom_est}
    Let $\varphi,\psi: \R^2 \to \R$ be such that $\varphi$ has compact support and $\psi$ is bounded and consider the centred, averaged field  $\widebar{Z}_{N}^{\beta_N}
        (\varphi,\psi)$ with respect to $\varphi,\psi$, as in \eqref{def:avg-field}. Let also $w:\R^2 \to \R$ be a weight function such that $\log w$ is Lipschitz continuous. Then, for every $h\in \N$ with $h \geq 3$, $\bht \in (0,1)$, there exist $\sfa_*=\sfa_*(h,\bht,w) \in (0,1)$ and $C=C(h,\bht,w) \in (0,\infty)$ such that for any $p,q \in (1,\infty)$ that satisfy $\frac{1}{p}+\frac{1}{q}=1$ and $p \, q \leq \sfa_* \log N$, the following inequality holds:
    \begin{align} \label{av_est}
        \bigg| \bbE \Big[	\widebar{Z}_{N}^{\beta_N}
            (\varphi,\psi)^h	\Big]	\bigg | \leq \Big(\frac{C\, p\,q }{\log N}\Big)^{\frac{h}{2}}\cdot \frac{1}{N^h}
        \cdot \norm{\frac{\varphi_N}{w_N}}^h_{\ell^p}\norm{\psi_N}^h_{\infty} \norm{w_N}_{\ell^q}^h \, ,
    \end{align}
    where for $x\in \Z^2$ we have $\varphi_N(x):=\varphi(x/\sqrt{N})$, $\psi_N(x):=\psi(x/\sqrt{N})$ and $w_N(x):=w(x/\sqrt{N})$.
    Moreover, 
    for $\widebar{Z}_{N}^{\beta_N}$ being the centred, point-to-line partition function, it holds that
    \begin{align}\label{onept_est}
       \sup_{N \in \N} \bigg| \bbE \Big[	\big(\widebar{Z}_{N}^{\beta_N}\big)^h\Big]
        \bigg | < \infty.
    \end{align}
\end{theorem}
Let us comment on estimate \eqref{av_est} and the significance of the precise dependence of the constant of  
the inequality \eqref{av_est} in terms of $p,q$ and $N$  as $(pq/\log N)^{h/2} N^{-h}$ (the generic constant $C$ that appears in the right-hand side
of \eqref{av_est} does not depend on either $p$, $q$ or $N$):
In the case that $\varphi, \psi, w$ are nice, smooth functions, i.e. the partition function field is averaged out in an essentially uniform way,
then, by Riemann summation, $N^{-h/p}\|\frac{\varphi_N }{w_N} \|_{\ell^p}^h$ will converge to $\|\varphi\|^h_{L^p}$ and 
$N^{-h/p}\|w_N \|_{\ell^q}^h$ will converge to $\|w\|^h_{L^q}$ and then the right hand side of \eqref{av_est} will capture the decay 
of correlations of the partition function field as $(\log N)^{-h/2}$.
In contrast to this scenario, if we want to estimate the moments of a {\it single} partition function, i.e. if we only restrict to 
a single starting point, say $Z^\beta_N(0)$,
 we would need to insert in \eqref{av_est} a {\it delta-like function} $\varphi_N(x):=N\ind_{\{ x=0 \}}$.
 This choice, however, leads to a blow up in $N$ in the right hand side
of \eqref{av_est} of the form $\big(C pq /(\log N)\big)^{h/2}\cdot N^{h/q}$. 
The idea in order to neutralise the blow up, is to optimise over the choice of $p,q$, 
of the corresponding $\ell^p$ and $\ell^q$ spaces by choosing 
 $q:=\sfa \log N$ (for certain parameter $\sfa$). 
 This is why the explicit dependence of the constant in the right-hand side in terms of $p,q$ is crucial
 and it is remarkable that this concrete dependence leads to a sharp estimate after a suitable optimisation. 
 The extraction of the precise dependence on $p,q$ in inequality \eqref{av_est} has been one of our main efforts in this work.

\smallskip

\subsection{Related literature.}
Let us explain how our methods and results fit the various aspects of the literature as well as outline some prospects.
\vskip 2mm
{\bf Moments, Functional inequalities and Schr\"odinger operators. }
Inequality \eqref{av_est} controls the moments of the partition function when starting and ending points are averaged out.
A version of  \eqref{av_est} but for $2d$ polymer at the critical temperature
 and without the important for our purposes dependence on $p,q$ was established in \cite{CSZ21}.
Such inequality was used there as an input to prove uniqueness of the scaling limit of the polymer
field at the critical temperature scaling.
In the continuum setting of the $2d$ stochastic heat equation (again at critical temperature scaling),
 existence of all moments of the fields averaged by test functions $\phi,\psi$, which are restricted 
 to only belong to $L^2(\R^2)$ was proved in \cite{GQT21}.  
 
 Important input in all the above came from
earlier works of Dell'Antonio-Figari-Teta \cite{DFT94} and Dimock-Rajeev \cite{DR04}
on the spectral theory of Schr\"odinger operators with point interactions. The objective there was to define,
via renormalisation, self-adjoint extensions of many-body hamiltonians of the form
\begin{align*}
\Delta + \sum_{1\leq i<j\leq h} \delta(x_i-x_j), \qquad x_1,...,x_h\in \R^2,
\end{align*}
with $\delta(\cdot)$ being the delta function on $\R^2$ and $\Delta$ the Laplacian on $(\R^2)^h$. 
This problem has a rich history, as outlined in \cite{DFT94}.
Of particular significance for our purposes is
a functional inequality of Dell'Antonio-Figari-Teta \cite{DFT94}, Proposition 3.1 (see also Lemma 5.1 in \cite{GQT21}),
which essentially states that the Green's function of a $2h-$dimensional Brownian motion  
$\mathfrak{G}(\bx,\by), \bx=(x_1,...,x_h)\in (\R^2)^h, \by=(y_1,...,y_h)\in (\R^2)^h$, 
when restricted to hyperplanes $\{x_i=x_j\}$ and $\{y_k=y_\ell\}$ with $(i,j)\neq (k,\ell)$
is bounded as an operator from $L^2(\R^{2(h-1)}) \to L^2(\R^{2(h-1)})$. 
For our purposes, the relevant operator, in the discrete setting is defined in \eqref{free_evol} and \eqref{LaplaceQ} and the 
corresponding estimate that establishes the boundedness of the operator in $\ell^q(\Z^{2(h-1)})$, $q>1$, is proved in Proposition \ref{Q_IJ_norm}. To prove this, we follow the approach of
\cite{CSZ21}, working on the real space rather than the Fourier space as \cite{DFT94} (the latter is only suitable
for $L^2$ estimates) paying particular attention in extracting the precise dependence on the parameters $p,q$
of the $\ell^p, \ell^q$ spaces. 
As already mentioned after Theorem \ref{mom_est}, this is crucial in order to apply
the optimisation trick, which allows us to average over a delta-like function and thus control the moment of
the, more singular, {\it single-starting point} polymer partition.

\vskip 2mm
{\bf Collision local times.} 
Theorem \ref{loc_times} led us to conjecture and prove in \cite{LZ22} the following multivariate extension of the Erd\"os-Taylor theorem:
\begin{theorem}[\cite{LZ22}]\label{thm:LZ22}
For $h\geq 3$ the joint law of the normalised pairwise collision local times $\Big\{\frac{\,\pi  \sfL^{(i,j)}_N\, }{\log N}\Big\}_{1\leq i<j \leq h}$
of $h$ independent, two-dimensional simple, symmetric random walks
converges in distribution to a vector of $\tfrac{h(h-1)}{2}$ independent {\rm Exp}$(1)$ random variables. 
\end{theorem}

The conjecture emerged from Theorem \ref{loc_times} and the fact
 that a gamma distributed variable $\Gamma(a,1), a\in \N,$ may arise as a sum of $a$ independent Exp$(1)$ random variables,
 The methods developed around Theorem \ref{mom_est} in the present paper played an important role towards the proof of Theorem \ref{thm:LZ22}
 as they paved the way for a number of necessary approximations. Still, in order to establish the independence of the collision times, one had
 to look carefully at the structure of the collisions: There is an intrinsic logarithmic scaling, already manifested in a sense in \eqref{thm:ET},
 which introduces a certain {\it separation of scales}.  In turn, this separation of scales leads to the phenomenon (called {\it rewiring} in \cite{LZ22})
  in which the random walks forget how they tracked collisions and behave as if they only followed the pairwise collisions independently.
  We refer to \cite{LZ22} for further details.
  
  Let us mention that a reader who has followed the literature might have thought initially that Theorem \ref{thm:LZ22}
  could simply follow from an extension of the work \cite{CSZ17b} to joint convergence of partition functions
  $Z_N^{\beta_{N,1}},...,Z_N^{\beta_{N,h}}$ at multiple temperatures $\beta_{N,i}=\beta_i\sqrt{\frac{\pi}{\log N}}$ for $ i=1,...,h$, 
  together with the moment estimates obtained here. However, this is not sufficient as the computation of 
  $\mathbb{E}[Z_N^{\beta_{N,1}} \cdots Z_N^{\beta_{N,h}}]$ gives rise only to a functional of the form 
  $\E^{\otimes h}[e^{\frac{\pi}{\log N}\sum_{1\leq i<j\leq h} \beta_i\beta_j \sfL_N^{(i,j)} }]$, while one would need to have $\binom{h}{2}$
  independent parameters $\beta_{i,j}, 1\leq i<j\leq h$ instead of just $\beta_i\beta_j$ to identify the joint distribution.

 Theorems \ref{loc_times} and \ref{thm:LZ22} create some interesting connections with phenomena from planar Brownian motion.
 In particular, related to windings of planar Brownian motion \cite{Y91} and more general, so-called {\it log-scaling laws} \cite{PY86}, \cite{Kn93}. A paradigm in these studies is the following:
 Let $B^1,...,B^h$ denote  $h$ independent planar Brownian motions starting from distinct points $z_1,\dots,z_h \in \R^2$ and for each pair $1 \leq i <j \leq h$, we consider $\theta^{(i,j)}(t)$ to be the total winding angle of $Z^{i,j}_s:=B_s^i-B_s^j$ around $0$. Yor's theorem \cite{Y91} asserts that $\big\{ \frac{2}{\log t}\, \theta^{(i,j)}(t) \big\}_{1\leq i<j \leq h}\xrightarrow[N \to \infty]{(d)} \big \{C^{(i,j)}\big \}_{1 \leq i<j \leq h}$, where $ \big \{C^{(i,j)}\big \}_{1 \leq i<j \leq h}$
 are $\tfrac{h(h-1)}{2}$ independent Cauchy distributed random variables. This result generalises to multiple Brownian motions the classical Spitzer's law
 \cite{S58} and bears resemblance to Theorem \ref{thm:LZ22}.
 The above works rely crucial on the continuum methods and the power of stochastic calculus. In this regard, 
 it would be interesting to investigate the scope of the methods we develop here (as well as in \cite{LZ22}) for wider applications 
and beyond continuum aspects.

\vskip 2mm
{\bf Statistics of log-correlated fields.} 
Let us denote by $\mathfrak{h}_N(x):=\sqrt{\log N} \big( \log Z^{\beta_N}_{N} \big( \lfloor x\sqrt N\rfloor 
-\bbE \log Z^{\beta_N}_{N} \big( \lfloor x\sqrt N\rfloor \big)$, for $x\in\R^2$.
It is known \cite{CSZ20, G20} that, in the subcritical regime, $\mathfrak{h}_N(x)$
converges to the gaussian free field. This, now, raises interesting questions \cite{CZ21}
about asymptotic statistics of extrema of the field $\mathfrak{h}_N(x)$ 
as well as whether the exponential of $\mathfrak{h}_N(x)$, normalised by its mean, converges to a 
Gaussian Multiplicative Chaos (GMC).
The activity in the field of log-correlated fields and on questions of this type for various models is very 
large, so we will only refer to some reviews \cite{B17, BK22, BP21, DRSV17, RV14} for further guidance.

Since $\exp\big(\sqrt{\log N} \log Z^{\beta_N}_{N}\Big) =  \big( Z^{\beta_N}_{N}\big)^{\sqrt{\log N}}$, questions like the above 
appear naturally related to information on the asymptotic of moments of the partition function of order $\sqrt{\log N}$.
Progress in this direction, has recently been made in \cite{CZ21} where the authors showed that there exists a $\hat\beta_0<1$
such that for $\hat\beta \leq \hat\beta_0$ and for $h=h_N$ such that
\begin{align*}
\limsup_{N\to\infty} \frac{3\hat\beta^2}{1-\hat\beta^2}\frac{1}{\log N} {\binom{h}{2}} <1,
\end{align*}
then $\mathbb{E}\big[ (Z^{\beta_N}_N)^h\big]\leq C \big( \frac{1}{1-\hat\beta^2}\big)^{\tfrac{h(h-1)}{2}(1+\varepsilon_N)}$,
for $C=C(\hat\beta)$ and $\varepsilon_N=\varepsilon_N(\hat\beta)\to 0$ for $N\to \infty$. 
Tighter estimates might be needed in order to tackle the above questions on extrema and approximation to GMC,
see discussion \cite{CZ21}, Section 4.


Finally, let us remark that in the continuum setting of the stochastic heat equation
\begin{align*}
\partial_t u_\epsilon =\frac{1}{2} \Delta u_\epsilon + \beta_\epsilon \,\xi^{(\epsilon)}(t,x) u_\epsilon,
\end{align*}
with $\xi^{(\epsilon)}(t,x):= \tfrac{1}{\epsilon} \int_{\R^2} j(\frac{x-y}{\epsilon}) \xi(t,x) \dd x$ the mollified noise
and $\beta_\epsilon=\frac{\hat\beta}{\sqrt{\log \epsilon^{-1}}}$ with $\hat\beta<\sqrt{2\pi}$, corresponding to the subcritical regime for the
stochastic heat equation, the recent work \cite{CSZ22} yields that  
$\bbE u_\epsilon^h \geq \big(\bbE u_\epsilon^2 \big)^{\frac{h(h-1)}{2}}$, for any positive $h$, irrespective
the dependence of $h$ in $\epsilon$ (one should think of a correspondence between $\epsilon$ and $N$ as $N=\epsilon^{-2}$ ).

\vskip 1cm
\subsection{Outline of the paper}
\begin{itemize}
    \item[--] In Section \ref{aux_tools} we set up the general framework
    (including a chaos and a renewal representation) and recall some
     results that will be useful for proving the main theorems.
    \item[--] In Section \ref{mom_expansion} we present the moment expansion and functional analytic framework. We also prove 
    the key operator norm estimates  in subsection \ref{integral_inequalities}.
    \item[--]In Section \ref{main_proofs} we present the proofs of Theorems  \ref{onept_mom}, \ref{loc_times}, \ref{mom_est} and \ref{avg_field_thm}.
    \item[--] Last, Appendix  \ref{technical_estimates} contains some technical estimates we make use of in Section \ref{mom_expansion}.
\end{itemize}

\smallskip

\section{Auxiliary tools} \label{aux_tools}
In this section we develop all the necessary machinery for the proof of the main results.
\subsection{Partition functions and chaos expansion}
Let us start by denoting the transition probability kernel of the underlying, two-dimensional, simple random walk $S$ by $q_n(x)$ for
$n\in \N$ and $x\in\Z^2$, that is $q_n(x):=\P(S_n=x)$. Recall from \eqref{pf_def} the definition of the \textit{point-to-plane} partition function
\begin{align*}
    Z^{\beta_N}_{N}(x):=\E_x \bigg[ e^{\,\sum_{n=1}^{N-1} \,\big(\beta_N \, \omega(n,S_n)-\lambda(\beta_N)\,\big)} \bigg] \, ,
\end{align*}
where $\beta_N$ is chosen so that
\begin{align}\label{def:beta-scale}
    \sigma^2_{N,\beta}:=e^{\lambda(2 \beta_N)-2\lambda(\beta_N)}-1=\frac{\hat\beta^2}{R_N},
\end{align}
where
\begin{align}\label{RN}
    R_N:=\E^{\otimes 2}\Big[ \sum_{n=1}^N \ind_{\{S_n^{(1)}=S_n^{(2)}\}}\Big]=\sum_{n=1}^N \sum_{z\in\Z^2} q_n(z)^2= \sum_{n=1}^N q_{2n}(0)\, ,
\end{align}
denotes the expected collisions until time $N$ of two independent, two-dimensional, simple random walks, starting from the origin.
Note that  \cite{ET60}
\begin{align*}
    R_N 	 = \frac{\log N}{\pi} + \frac{\alpha}{\pi} + o(1) \, ,
\end{align*}
where $\alpha := \gamma + \log 16 - \pi \simeq 0.208$ and
$\gamma  \simeq 0.577 $ is the Euler constant.
By Taylor expansion in \eqref{def:beta-scale}, this implies
the asymptotic scaling of $\beta_N$ as $\beta_N \sim \bht \sqrt{\frac{\pi}{\log N}}$ for $N\to\infty$.

We shall also need the definition of the \textit{point-to-point} partition functions. In particular, for $a,b \in \N$ with $a<b$ and $x,y \in \Z^2$,
we define the \textit{point-to-point} partition function from the space-time point $(a,x)$ to $(b,y)$ by
\begin{align} \label{ptp_def}
    Z^{\beta_N}_{a,b}(x,y):=\E_{a,x} \bigg[ e^{\,\sum_{n=a+1}^{b-1} \,\big(\beta_N \, \omega(n,S_n)-\lambda(\beta_N)\,\big)}\, \ind_{\{S_b=y\}} \bigg] \, ,
\end{align}
Note that with these definitions,
\begin{align*}
    Z^{\beta_N}_{N}(x)=\sum_{y \in \Z^2} Z^{\beta_N}_{0,N}(x,y) \, .
\end{align*}
Given $\varphi,\psi: \R^2 \to \R$ such that $\varphi$ has compact support and $\psi$ is bounded, we can further define the \textit{averaged partition functions} by,
\begin{align*}
     & Z^{\beta_N}_{a,b}(\varphi,y):=\sum_{x \in \Z^2} \varphi(\tfrac{x}{\sqrt{N}})\, Z^{\beta_N}_{a,b}(x,y)\, , \\
     & Z^{\beta_N}_{a,b}(x,\psi):=\sum_{y \in \Z^2} \, Z^{\beta_N}_{a,b}(x,y)\, \psi(\tfrac{y}{\sqrt{N}})
\end{align*}
and
\begin{align} \label{avg_field_def}
    Z^{\beta_N}_{a,b}(\varphi,\psi):=\frac{1}{N} \sum_{x,y} \varphi(\tfrac{x}{\sqrt{N}})\,Z^{\beta_N}_{a,b}(x,y)\, \psi(\tfrac{y}{\sqrt{N}}) \, .
\end{align}
For $(a,x),(b,y) \in \N\times \Z^2$ with $a<b$, the mean of each of the quantities above is computed as
\begin{equation} \label{q_phi}
    \begin{split}
     & \bbE\big[Z^{\beta_N}_{a,b}(\varphi,y)\big]=\, q^{N}_{a,b}(\varphi,y):=\sum_{x \in \Z^2} \varphi(\tfrac{x}{\sqrt{N}}) \, q_{a,b}(x,y) \, , \\
     & \bbE\big[Z^{\beta_N}_{a,b}(x,\psi)\big]=\,q^{N}_{a,b}(x,\psi):=\sum_{y \in \Z^2}  \, q_{a,b}(x,y)\, \psi(\tfrac{y}{\sqrt{N}}) \,
    \end{split}
\end{equation}
and
\begin{align*}
    \bbE\big[Z^{\beta_N}_{a,b}(\varphi,\psi)\big]=q^{N}_{a,b}(\varphi,\psi):= \frac{1}{N} \sum_{x,y \in \Z^2} \varphi(\tfrac{x}{\sqrt{N}}) \, q_{a,b}(x,y)\psi(\tfrac{y}{\sqrt{N}}) \, .
\end{align*}
Next, we derive an expansion for the point-to-point partition function $Z^{\beta_N}_{a,b}(x,y)$ as a multilinear polynomial, which goes by the name of \textit{chaos expansion}. This is the starting point of our analysis. Recalling \eqref{ptp_def} we have
\begin{align*}
    \qquad Z^{\beta_N}_{a,b}(x,y)= \E_{a,x}\bigg[ \prod_{a<n<b} \prod_{z \in \Z^2} e^{\big(\,\beta_N \, \omega(n,z) -\lambda(\beta_N)\,\big)\,\ind_{\{S_n=z\}}}\,\ind_{\{S_b=y\}}\bigg]
\end{align*}
and by using the fact that for $\lambda \in \R$, $e^{\lambda \, \ind_{\{S_n=z\}}}=1+(e^{\lambda}-1)\ind_{\{S_n=z\}}$ we obtain
\begin{align} \label{ptp_prod}
    Z^{\beta_N}_{a,b}(x,y)= \E_{a,x}\bigg[ \prod_{a<n<b} \prod_{z \in \Z^2} \big(1+\xi(n,z) \ind_{\{S_n=z\}}\big)\, \ind_{\{S_b=y\}} \bigg] \,
\end{align}
where $\xi(n,z):=e^{\beta_N \,\omega(n,z)-\lambda(\beta_N)}-1$ are i.i.d. random variables with
\begin{align} \label{xi_prp}
    \bbE[\xi]=0 \, ,\qquad \bbE[\xi^2]= e^{\lambda(2\beta_N)-2 \lambda(\beta_N)}-1=:\sigma^2_{N,\bht} \stackrel{N \to \infty}{\sim} \beta_N^2 \, , \quad  
    \bbE\big[|\xi|^k\big] \leq C_k \,\sigma^k_{N,\bht} \, \quad \text{for } k \geq 3 \, ,
\end{align}
for some constants $C_k \in (0,\infty)$, $k \geq 3$.
The asymptotic and the bound in \eqref{xi_prp}  follow by Taylor expansion.
Expanding the product in \eqref{ptp_prod} yields the following expansion of $Z^{\beta_N}_{a,b}(x,y)$ as a multilinear polynomial of the variables $\xi(n,z)$,
\begin{align} \label{ptp_exp}
    Z^{\beta_N}_{a,b}(x,y) & =q_{a,b}(x,y)\notag                                                                                                                                                                                   \\
                           & +\sum_{k \geq 1} \sumtwo{a<n_1<\dots<n_k<b}{z_1,\dots,z_k \in \Z^2} q_{a,n_1}(x,z_1)\,\xi(n_1,z_1)\,\bigg \{ \prod_{j=2}^k q_{n_{j-1},n_j}(z_{j-1},z_j)\, \xi(n_j,z_j) \bigg\}\,q_{n_k,b}(z_k,y) \, ,
\end{align}
which also leads to
\begin{align*}
    Z^{\beta_N}_{a,b}(\varphi,\psi) & :=q_{a,b}^N(\varphi,\psi)\notag \\ &+ \frac{1}{N} \sum_{k \geq 1} \sumtwo{a<n_1<\dots<n_k<b}{z_1,\dots,z_k \in \Z^2} q^N_{a,n_1}(\varphi,z_1)\,\xi(n_1,z_1)\,\bigg \{ \prod_{j=2}^k q_{n_{j-1},n_j}(z_{j-1},z_j)\, \xi(n_j,z_j) \bigg\}\,q^N_{n_k,b}(z_k,\psi)
\end{align*}
for the averaged  point-to-point partition function. Using the notation
\begin{align*}
    \widebar{Z}^{\beta_N}_N(\varphi,\psi):=Z^{\beta_N}_N(\varphi,\psi)- \bbE\big[Z^{\beta_N}_N(\varphi,\psi)\big]
\end{align*}
for the centred averaged partition function we have that
\begin{align} \label{cent_exp}
    \widebar{Z}^{\beta_N}_{N}(\varphi,\psi):=\frac{1}{N}\sum_{k \geq 1} \sumtwo{z_1,z_2,\dots,z_k}{0<n_1<\dots<n_k< N} \hspace{-0.3cm} q^{N}_{0, n_1}(\varphi,z_1)\, \xi(n_1,z_1) \bigg\{ \prod_{j=2}^k q_{n_{j-1},n_j}(z_{j-1},z_j) \xi(n_j,z_j) \bigg\} q^{N}_{n_k,N}(z_k,\psi) \, .
\end{align}
For simplicity, we will denote $Z^{\beta_N}_{N}(\varphi):=Z^{\beta_N}_{N}(\varphi,1)$ and
$\widebar{Z}^{\beta_N}_{N}(\varphi):=\widebar{Z}^{\beta_N}_{N}(\varphi,1)$.

\subsection{Renewal representation}\label{sec:renewal}
We will also need certain renewal representations for the second moment of the point-to-point partition functions.
These were introduced  in \cite{CSZ19b} but only mainly studied in the context of the critical directed polymer therein.
Let $(a,x),(b,y) \in \N \times \Z^2$ with $a<b$. We define
\begin{align} \label{un_def}
    U^{\beta_N}_N\big((a,x),(b,y)\big):=\sigma_{N,\bht}^2 \, \bbE\Big[Z^{\beta_N}_{a,b}(x,y)^2\Big] \, .
\end{align}
By translation invariance
\begin{align*}
    U^{\beta_N}_N\big((a,x),(b,y)\big)=U^{\beta_N}_N(b-a, y-x)
    := \sigma_{N,\bht}^2 \,\bbE\Big[Z^{\beta_N}_{b-a}(y-x)^2\Big],
\end{align*}
therefore it suffices to work with $U^{\beta_N}_N(n,x)$. We furthermore define $U_N^{\beta_N}(n,x):=\ind_{\{x=0\}}$ if $n=0$.
Using \eqref{ptp_exp} and \eqref{xi_prp} we derive the expansion
\begin{align} \label{Un_exp}
    U^{\beta_N}_N(n,x) & =\sigma^2_{N,\bht}\, q^2_{n}(x)                                                                                                                                                                                 \\
                       & + \sum_{k \geq 1} \sigma^{2(k+1)}_{N,\bht} \sumtwo{0<n_1<\dots<n_k<n}{z_1,z_2,\dots,z_k \in \Z^2}  q^2_{0,n_1}(0,z_1)\Big\{ \prod_{j=2}^k q^2_{n_{j-1},n_j}(z_{j-1},z_j) \Big\}\, q^2_{n_k,n}(z_k,x) \, .\notag
\end{align}
Moreover, for $0\leq n \leq N$ we define
\begin{align} \label{Un_sumx}
    U^{\beta_N}_N(n):=\sum_{x \in \Z^2} U^{\beta_N}_N(n,x) \, .
\end{align}
We will, now, recast $U^{\beta_N}_N(n,x)$ and $U^{\beta_N}_N(n)$ in a renewal theory framework.
We define a family of i.i.d. random vectors $(\sft^{(N)}_i, \sfx^{(N)}_i)_{i \geq 1}$, such that
\begin{align*}
    \P\Big(\,\big(\sft^{(N)}_1, \sfx^{(N)}_1 \big)=(n,x) \Big)=\frac{q_n^2(x)}{R_N}\,\ind_{\{n \leq N\}} \,
\end{align*}
and moreover we let $\tau^{(N)}_k:=\sft_1^{(N)}+\dots+\sft_k^{(N)}$ and $S^{(N)}_k:=\sfx^{(N)}_1+\dots+\sfx^{(N)}_k$ if $k \geq 1$. For $k=0$ we set $(\tau_0,S_0):=(0,0)$. Using this framework we see by \eqref{Un_exp} and \eqref{Un_sumx} that
\begin{align*}
    U^{\beta_N}_N(n,x)=\sum_{k\geq 0} \bht^{2k} \, \P\big(\tau_k^{(N)}=n, S^{(N)}_k=x\big)
\end{align*}
and
\begin{align*}
    U^{\beta_N}_N(n)=\sum_{k\geq 0} \bht^{2k} \, \P\big(\tau_k^{(N)}=n\big) \, .
\end{align*}
Finally, we remark that
\begin{align} \label{Uneqvar}
    \sum_{n=0}^N U_N^{\beta_N}(n)=\bbE\Big[(Z^{\beta_N}_{N+1})^2\Big]  \, .
\end{align}

\subsection{Some useful results}
We will make use of the following results on the limiting distribution of $Z_{N}^{\beta_N}$ and the fluctuations of $\widebar{Z}_{N}^{\beta_N}
    (\varphi)$, which were established in \cite{CSZ17b}.
\begin{theorem} \label{TheoremA}\cite{CSZ17b}.
    Fix $\bht \in (0,1)$ and let $\rho^2_{\bht}:=\log\big(\frac{1}{1-\bht^2}\big)$.
    Then,
    \begin{align*}
        Z_{N}^{\beta_N}
        \xrightarrow[N \to \infty]{(d)} \exp\big(\rho_{\bht}\, \sfX -\tfrac{1}{2}\, \rho^2_{\bht}\big) \, ,
    \end{align*}
    where $\sfX$ has a standard normal distribution $\mathcal{N}(0,1)$.
\end{theorem}
\begin{theorem} \label{TheoremB}\cite{CSZ17b}.
    Fix $\bht \in (0,1)$ and $\varphi \in C_c(\R^2)$. Then,
    \begin{align*}
        \sqrt{\log N} \,  \widebar{Z}_{N}^{\beta_N}
        (\varphi) \xrightarrow[N \to \infty]{(d)} \cN\big(0,\rho^2_{\varphi}(\bht)\big) \, ,
    \end{align*}
    where  $\widebar{Z}_{N}^{\beta_N}
        (\varphi):=\widebar{Z}_{N}^{\beta_N}
        (\varphi,1)$ is defined in \eqref{avg_field_def},
    \begin{align*}
        \rho^2_{\varphi}(\bht):= \frac{\pi \,\bht^2}{1-\bht^2} \int_{0}^1 \dd t \int_{(\R^2)^2} \dd x \, \dd y \,\varphi(x) g_t(x-y) \varphi(y)
    \end{align*}
    and $g_t(x):=\frac{1}{2 \pi t} \, e^{-|x|^2/2t}$ denotes the two-dimensional heat kernel.
\end{theorem}

\section{Expansion of moments and integral inequalities} \label{mom_expansion}
We shall hereafter use the notation
\begin{align*}
    M^{\varphi,\psi}_{N,h}:= \bbE \Big[	\widebar{Z}^{\beta_N}_{N}(\varphi,\psi)^h	\Big] \, ,
\end{align*}
for the $h^{th}$ centred moments of the averaged field \eqref{avg_field_def}.
\subsection{Chaos expansion of moments}
By \eqref{cent_exp} we have
\begin{align} \label{Exp_h}
    M^{\varphi,\psi}_{N,h} & =  \frac{1}{N^h  }   \notag                           \\
                           & \, \times \bbE \Bigg [ \bigg(\sum_{k \geq 1} \sumtwo{z_1,z_2,\dots,z_k\in \Z^2}{0<n_1<\dots<n_k< N} q^N_{0, n_1}(\varphi,z_1)\xi(n_1,z_1) \bigg\{ \prod_{j=2}^k q_{n_{j-1},n_j}(z_{j-1},z_j) \xi(n_j,z_j) \bigg\} q^N_{n_k,N}(z_k,\psi)\bigg)^h \Bigg] \, .
\end{align}
When $h \in \N$, the power $h$ on the right hand side of \eqref{Exp_h} can be expanded as
\begin{align} \label{h_power_exp}
    \sum_{k_1,\dots,k_h\geq 1} \sumtwo{(n^{(r)}_i,z^{(r)}_i) \in \, \N \times \Z^2\, }{1\leq i \leq k_r\, , 1\leq r \leq h} & \prod_{r=1}^h q^N_{0, n^{(r)}_1}(\varphi,z^{(r)}_1)\xi(n^{(r)}_1,z^{(r)}_1)\,
    \notag                                                                                                                                                                                                  \\  \times & \bigg\{ \prod_{j=2}^{k_r} q_{n^{(r)}_{j-1},n^{(r)}_j}(z^{(r)}_{j-1},z^{(r)}_j) \xi(n^{(r)}_j,z^{(r)}_j) \bigg\} q^N_{n^{(r)}_k,N}(z^{(r)}_{k_r},\psi) \, .
\end{align}
Note that every term in that expansion contains a product of disorder variables of the form
\begin{align} \label{disorder_variables_product}
    \prod_{r=1}^h \, \prod_{j=1}^{k_r} \, \xi(n^{(r)}_j,x^{(r)}_j) \, .
\end{align}
Therefore,
after taking the expectation with respect to the environment and taking into account that the $\xi$ variables have mean zero and are independent if they are indexed by different space time points, see \eqref{xi_prp}, we see that the non-zero terms of the expansion of \eqref{Exp_h} will be those such that
for every point $(n^{(r)}_j,x^{(r)}_j)$, $1\leq j \leq k_r, 1\leq r \leq h$ there exists (at least one) $1\leq r' \leq h, 1\leq j \leq k_{r'}$ such that $r \neq r'$ and
$(n^{(r)}_{j},x^{(r)}_{j})=(n^{(r')}_{j'},x^{(r')}_{j'})$, that is, every disorder variable $\xi(n^{(r)}_j,x^{(r)}_j)$ should appear at least twice in a product of disorder variables. Hence, a natural way to parametrise the sum \eqref{Exp_h} is to sum over the space-time locations of these coincidence points along with all the possible coincidence configurations. We will also use iteratively the Chapman-Kolmogorov equation $q_{t_1,t_2}(x,y)=\sum_{z \in \Z^2}q_{t_1,s}(x,z)\, q_{s,t_2}(z,y)$, $t_1<s<t_2$, for the simple random walk, to break down 'long range jumps', appearing in \eqref{h_power_exp} via their transition probabilities, into smaller jumps, so that we can track the location of each random walk at each time $t$, see Figure \ref{fig:4thmom}. Let us introduce the framework which will allow to formalise the above.
\begin{figure}
    \tikzstyle{filled_vertex}=[fill={black}, draw={black}, shape=circle,inner sep=2]
    \tikzstyle{new style 0}=[fill=white, draw={black}, shape=circle, inner sep=2pt]
    \tikzstyle{straight_line}=[-, fill=none, draw={black},line width=0.9pt]
    \tikzstyle{wiggly_line}=[-, fill=none, draw={black}, decorate=true, decoration={snake,amplitude=.4mm,segment length=2mm},line width=0.9pt]

    \tikzstyle{none}=[inner sep=0pt]
    \pgfdeclarelayer{edgelayer}
    \pgfdeclarelayer{nodelayer}
    \pgfsetlayers{edgelayer,nodelayer,main}
    \usetikzlibrary{decorations.pathmorphing}

    \begin{tikzpicture}[scale=0.5]

        \begin{pgfonlayer}{nodelayer}
            \node [style=none] (0) at (-17, 8) {};
            \node [style=none] (1) at (-17, -4) {};
            \node [style=none] (2) at (-12, -4) {};
            \node [style=none] (3) at (-12, 8) {};
            \node [style=none] (4) at (-9, -4) {};
            \node [style=none] (5) at (-9, 8) {};
            \node [style=none] (6) at (-5, -4) {};
            \node [style=none] (7) at (-5, 8) {};
            \node [style=none] (8) at (-1, 8) {};
            \node [style=none] (9) at (-1, -4) {};
            \node [style=none] (10) at (2, -4) {};
            \node [style=none] (11) at (2, 8) {};
            \node [style=none] (12) at (5, -4) {};
            \node [style=none] (13) at (5, 8) {};
            \node [style={filled_vertex}] (14) at (-15, 2) {};
            \node [style=none] (15) at (-12, 2) {};
            \node [style=none] (17) at (-12, -1) {};
            \node [style={filled_vertex}] (19) at (-12, 2) {};
            \node [style=new style 0] (21) at (-12, 4.5) {};
            \node [style=new style 0] (22) at (-12, -1) {};
            \node [style=new style 0] (24) at (-9, 4.75) {};
            \node [style=new style 0] (25) at (-5, 4) {};
            \node [style=new style 0] (26) at (-9, 2) {};
            \node [style=new style 0] (27) at (-5, 2) {};
            \node [style={filled_vertex}] (28) at (-1, 3) {};
            \node [style=new style 0] (29) at (2, 6) {};
            \node [style=new style 0] (30) at (5, 7) {};
            \node [style=none] (31) at (7, 8) {};
            \node [style=none] (32) at (7, -4) {};
            \node [style=new style 0] (33) at (-1, -2) {};
            \node [style=new style 0] (34) at (2, -1) {};
            \node [style=new style 0] (35) at (5, -2) {};
            \node [style=new style 0] (36) at (7, -1) {};
            \node [style=new style 0] (37) at (7, 6) {};
            \node [style={filled_vertex}] (38) at (-9, 0) {};
            \node [style={filled_vertex}] (39) at (-5, 0) {};
            \node [style={filled_vertex}] (40) at (2, 3) {};
            \node [style={filled_vertex}] (41) at (2, 3) {};
            \node [style={filled_vertex}] (42) at (5, 3) {};
            \node [style=new style 0] (43) at (7, 4) {};
            \node [style=new style 0] (44) at (7, 2) {};
            \node [style=new style 0] (45) at (-17, 2) {};
            \node [style=none] (46) at (-15, 8) {};
            \node [style=none] (47) at (-15, -4) {};
            \node [style=new style 0] (48) at (-15, 3.75) {};
            \node [style=new style 0] (49) at (-15, -1) {};
            \node [style=none] (51) at (-17.75, 2) {};
            \node [style=none] (52) at (-17, 9) {};
            \node [style=none] (53) at (-15, 9) {$\scriptstyle I_1$};
            \node [style=none] (54) at (-12, 9) {$\scriptstyle I_2$};
            \node [style=none] (55) at (-9, 9) {$\scriptstyle I_3$};
            \node [style=none] (56) at (-5, 9) {$\scriptstyle I_4$};
            \node [style=none] (57) at (-1, 9) {$\scriptstyle I_5$};
            \node [style=none] (58) at (2, 9) {$\dots$};
            \node [style=none] (59) at (5, 9) {};
            \node [style=none] (60) at (7, 9) {};
            \node [style=none] (62) at (-14, 2.5) {$ \scriptstyle y_2^1=y_3^1$};
            \node [style=none] (63) at (-11, 2.5) {$\scriptstyle y_2^2=y_3^2$};
            \node [style=none] (64) at (-11.5, 5.25) {$\scriptstyle y_1^2$};
            \node [style=none] (65) at (-8, 0.5) {$\scriptstyle y_3^3=y_4^3$};
            \node [style=none] (66) at (-8.5, 2.5) {$\scriptstyle y_2^3$};
            \node [style=none] (67) at (-8.5, 5.25) {$\scriptstyle y_1^3$};
            \node [style=none] (68) at (-11.5, -0.25) {$\scriptstyle y_4^2$};
            \node [style=none] (69) at (-14.5, -0.25) {$\scriptstyle y_4^1$};
            \node [style=none] (70) at (-14.5, 4.5) {$\scriptstyle y_1^1$};
            \node [style=none] (71) at (-4.5, 4.5) {$ \scriptstyle y_1^4$};
            \node [style=none] (72) at (-4.5, 2.75) {$\scriptstyle y_2^4$};
            \node [style=none] (73) at (-6, -0.75) {$\scriptstyle y_3^4=y^4_4$};
            \node [style=none] (74) at (0.45, 2) {$\scriptstyle y_1^5=y_2^5=y_3^5$};
            \node [style=none] (75) at (-0.5, -2.5) {$\scriptstyle y_4^5$};
            \node [style=none] (76) at (2.5, -0.5) {};
            \node [style=none] (77) at (2.5, 3.5) {};
            \node [style=none] (78) at (2.5, 6.75) {};
            \node [style=none] (79) at (5.75, 7.25) {};
            \node [style=none] (80) at (7.5, 6.5) {};
            \node [style=none] (81) at (7.5, 4.5) {};
            \node [style=none] (82) at (7.5, 2.5) {};
            \node [style=none] (83) at (7.5, -0.5) {};
            \node [style=none] (84) at (5.5, -2.25) {};
        \end{pgfonlayer}
        \begin{pgfonlayer}{edgelayer}
            \draw [style={straight_line}] (0.center) to (1.center);
            \draw [style={straight_line}] (3.center) to (2.center);
            \draw [style={straight_line}] (5.center) to (4.center);
            \draw [style={straight_line}] (7.center) to (6.center);
            \draw [style={straight_line}] (8.center) to (9.center);
            \draw [style={straight_line}] (10.center) to (11.center);
            \draw [style={straight_line}] (13.center) to (12.center);
            \draw [style={straight_line}] (31.center) to (32.center);
            \draw [style={straight_line}] (46.center) to (47.center);
            \draw [style={straight_line}] (19) to (26);
            \draw [style={straight_line}] (26) to (27);
            \draw [style={straight_line}] (21) to (24);
            \draw [style={straight_line}] (24) to (25);
            \draw [style={straight_line}] (22) to (38);
            \draw [style={straight_line}] (19) to (38);
            \draw [style={straight_line}] (39) to (28);
            \draw [style={straight_line}] (27) to (28);
            \draw [style={straight_line}] (25) to (28);
            \draw [style={straight_line}] (28) to (29);
            \draw [style={straight_line}] (29) to (30);
            \draw [style={straight_line}] (30) to (37);
            \draw [style={straight_line}] (39) to (33);
            \draw [style={straight_line}] (33) to (34);
            \draw [style={straight_line}] (34) to (35);
            \draw [style={straight_line}] (35) to (36);
            \draw [style={straight_line}] (42) to (43);
            \draw [style={straight_line}] (42) to (44);
            \draw [style={wiggly_line}] (38) to (39);
            \draw [style={wiggly_line}] (14) to (19);
            \draw [style={wiggly_line}] (41) to (42);
            \draw [style={straight_line}, bend left=15, looseness=1.25] (28) to (41);
            \draw [style={straight_line}, bend right=15, looseness=1.25] (28) to (41);
            \draw [style={straight_line}, bend left=15, looseness=1.25] (45) to (14);
            \draw [style={straight_line}, bend right=15] (45) to (14);
            \draw [style={straight_line}] (45) to (48);
            \draw [style={straight_line}] (48) to (21);
            \draw [style={straight_line}] (45) to (49);
            \draw [style={straight_line}] (49) to (22);
        \end{pgfonlayer}
    \end{tikzpicture}
    \caption{ A diagrammatic representation of the expansion \eqref{exp_with_pairing} for $\bbE\big[(\widebar{Z}_{N}^{\beta_N})^4\big ]$. The horizontal direction is the time direction, while the vertical lines correspond to different time slices, $\{n\}\times \Z^2$, $n \in \N$. We use straight lines to represent free evolution \eqref{free_evol}
        and wiggly lines to represent replica evolution, see  \eqref{replica_op}. We use filled dots to represent space-time points where disorder $\xi$ is sampled.}
    \label{fig:4thmom}
\end{figure}
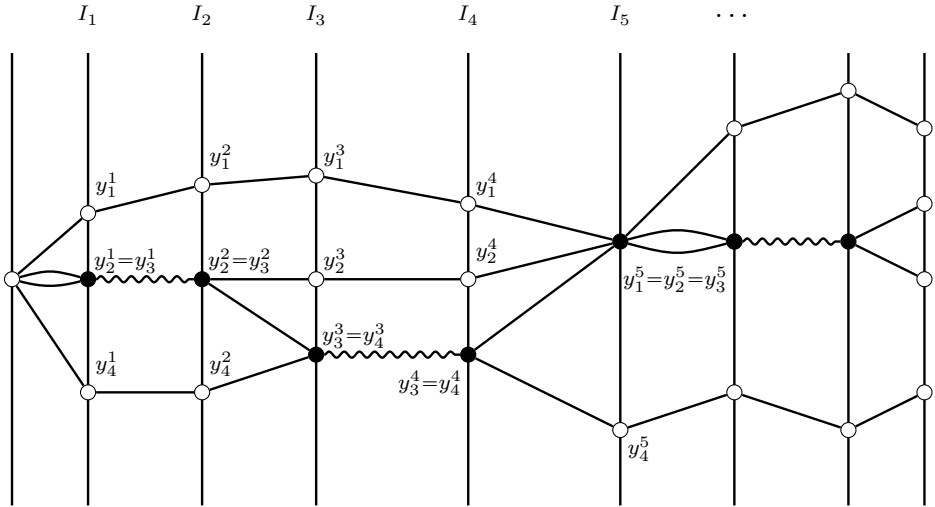

For $h \geq 3$, let $I \vdash \{1,\dots,h\}$ denote a partition $I=I_1\sqcup I_2 \sqcup \cdots \sqcup I_m$ of $\{1,\dots,h\}$ into disjoint subsets $I_1,\dots,I_m$ with cardinality $|I|=m$. Given $I \vdash \{1,\dots,h\} $, we define the equivalence relation $\stackrel{I}{\sim} $ such that for $k,\ell \in \{1,\dots,h\}$, we have $k\stackrel{I}{\sim}  \ell$ if $k$ and $\ell$ belong to the same component of the partition $I$.
For $\bx=(x_1,\dots,x_h)\in (\Z^2)^{h}$ and a partition $I$ we will denote $\bx \sim I$ if $x_k=x_\ell$ for all $k\stackrel{I}{\sim}  \ell$. We shall also use the notation $(\Z^2)^h_I:=\{\bx\in (\Z^2)^h \colon \bx\sim I\}$.

For $p \in(1,\infty)$ we define the $I$-restricted $\ell^p$ spaces $\ell^p\big((\Z^2)^h_I\big)$ via the norm
$\|f\|_{\ell^p((\Z^2)^h_I)}:=\big(\sum_{x\in (\Z^2)^h_I} |f(x)|^p \big)^{1/p}$ for functions $f\colon (\Z^2)^h_I \to \R$.
In shorthand, we will often write $\ell^p_I$ or just $\ell^p$ if there is no risk of confusion.
For an integral operator $\sfT:\ell^q\big((\Z^2)^h_J\big)\rightarrow \ell^q\big((\Z^2)^h_I\big)$, we define the pairing
\begin{align} \label{pairing}
    \langle f,\sfT g\rangle:=\sum_{\bx \,\in (\Z^2)^h_I,\by \,\in (\Z^2)^h_J} f(\bx)\sfT(\bx,\by)g(\by) \, .
\end{align}
The operator norm will be given by
\begin{align}\label{op_norm}
    \norm{\sfT}_{\ell^q \to \ell^q}:= \sup_{\norm{g}_{\ell^q_J} \leq 1} \norm{\sfT g}_{\ell^q_I}
    =\sup_{\norm{f}_{\ell^p_I}\leq 1, \, \norm{g}_{\ell^q_J}\leq 1} \langle f,\sfT g \rangle
\end{align}
for $p,q \in (1,\infty)$ conjugate exponents, i.e. $\frac{1}{p}+\frac{1}{q}=1$.

For two partitions $I,J\vdash \{1,\dots,h\}$ and $\bx,\by \in (\Z^2)^h$ with $\bx\sim I$ and $\by\sim J$
we define the {\it free evolution subject to constraints} $I,J$ as
\begin{align} \label{free_evol}
    Q^{I,J}_n(\bx,\by):=\ind_{\{\bx\sim I\}} \, \prod_{i=1}^h q_n(y_i-x_i) \, \ind_{\{\by \sim J\}}\, ,  \qquad \text{for $n\in \N$} \, .
\end{align}
$Q^{I,*}_n$ and $Q^{*,J}_n$ will denote the particular cases where $I$ and $J$, respectively,  are the partitions consisting only of singletons, i.e. $I=\{1\} \sqcup \cdots \sqcup \{h\}$. Moreover, if $I,J\vdash \{1,\dots,h\}$, $\varphi,\psi:\R^2\to \R$ and $n \in \N$ we define
\begin{equation} \label{Q_n_phi}
    \begin{split}
        Q_n^{*,J}(\varphi^{\otimes h},\by)&:=\prod_{i=1}^h q^N_n(\varphi,y_i) \cdot \ind_{\{\by \sim J\}} \\
        Q_n^{I,*}(\bx,\psi^{\otimes h})&:=\ind_{\{\bx \sim I\}} \cdot \prod_{i=1}^h q^N_n(x_i,\psi)  \, ,
    \end{split}
\end{equation}
see also \eqref{q_phi}.
The mixed moment subject to a partition $I$ will be denoted by
\begin{align} \label{mixed_mom}
    \bbE[\xi^I]:=\prod_{1\leq j \leq |I|, \, |I_j|\geq 2} \bbE[\xi^{|I_j|}] \, .
\end{align}
Using this formalism, we can then write
\begin{align} \label{MNh_expansion}
    M^{\varphi,\psi}_{N,h}= & \frac{1}{N^h} \sum_{k\geq 1} \sumthree{0:=n_0<n_1<\dots<n_k\leq N,}{(I_1,\dots,I_k)  \in \mathcal{I},}{m_i:=|I_i|<h, \,\, \by_i \in (\Z^2)^{m_i}} Q_{n_1}^{*,I_1}(\varphi^{\otimes h}, \by_1) \bbE\Big[\xi^{I_1}\Big] \notag \\
                            & \qquad  \qquad \qquad \times\prod_{i=2}^k Q_{n_i-n_{i-1}}^{I_{i-1},I_i}(\by_{i-1}, \by_i) \bbE\Big[\xi^{I_i}\Big]  \cdot 
                            Q_{N-n_k}^{I_k,*}(\by_k,\psi^{\otimes h}) \, ,
\end{align}
where $\mathcal{I}$ is the set of all finite sequences of partitions of $\{1,\dots,h\}$, $(I_1,\dots,I_k)$, which satisfy the following condition:
For every $r \in \{1,\dots,h\}$ there exists $1 \leq i \leq k$ such that the block
 $I_i$ that contains $r$ is non-trivial, i.e. it has cardinality equal or larger than $2$. 
This restriction comes from the fact that $M^{\varphi,\psi}_{N,h}$ are centred moments and therefore every term in the expansion \eqref{Exp_h} contains the expected value of a product of disorder variables \eqref{disorder_variables_product}, which is non-zero only if the product of disorder variables does not contain standalone $\xi$ variables, see the discussion below \eqref{disorder_variables_product}.

Let $B=B(0,r)\subset \R^2$ be a ball containing the support of $\psi$ (allowing the possibility of $r=\infty$, in case $\supp \psi=\R^2$). We then have that 
\begin{equation*}
    Q^{I_k,*}_{N-n_k}\big(\by_k,\psi^{\otimes h}\big)\leq Q^{I_k,*}_{N-n_k}\big(\by_k,\norm{\psi}^h_{\infty}\ind^{\otimes h}_{ B}\big) \leq \frac{c}{N}\, \sum_{n_{k+1} \in \{N+1,\dots,2N\}} Q^{I_k,*}_{n_{k+1}-n_k}\big(\by_k,\norm{\psi}^h_{\infty}\ind^{\otimes h}_{B}\big) \, ,
\end{equation*}
with the latter inequality following because the probability that a random walk starting inside the ball $B(0,\sqrt{N}r)\subset \R^2$ at time $N-n_k$ is inside $B(0,\sqrt{N}r)$ at time $n_{k+1}-n_k$ with $n_{k+1} \in \{N+1,\dots 2N\}$ is uniformly bounded away from zero.

Thus,
\begin{align} \label{up_bnd_exp}
    \big|\,M^{\varphi,\psi}_{N,h}\, \big| \leq
     & \, \frac{c\norm{\psi}^h_{\infty}}{N^{h+1}} \sum_{k\geq 1} \sumthree{0:=n_0<n_1<\dots<n_{k+1}\leq 2N,}{(I_1,\dots,I_k)  \in \mathcal{I},}{m_i:=|I_i|<h, \,\, \by_i \in (\Z^2)^{m_i}} Q_{n_1}^{*,I_1}(\varphi^{\otimes h}, \by_1) \bbE\Big[|\xi|^{I_1}\Big] \notag \\
     & \qquad \qquad  \qquad \times \prod_{i=2}^{k}
    Q_{n_i-n_{i-1}}^{I_{i-1},I_i}(\by_{i-1}, \by_i) \bbE\Big[|\xi|^{I_i}\Big] \cdot Q_{n_{k+1}-n_k}^{I_k,*}(\by_k,\ind^{\otimes h}_{ B}) \, .
\end{align}
We also need to define the {\it replica evolution}. For $I\vdash\{1,\dots,h\}$ of the form $I=\{k,\ell\} \sqcup \bigsqcup_{j\neq k,\ell}\{j\}$
\begin{align}\label{replica_op}
    \sfU^{I}_n(\bx,\by) := \ind_{\{\bx,\by\sim I\}}  \cdot U^{\beta_N}_{N}(n,y_k-x_k) \cdot \prod_{i\neq k,\ell} q_n(y_i-x_i) \, ,
\end{align}
where $U^{\beta_N}_{N}(n,y_k-x_k)$ is defined in \eqref{un_def}.  The replica evolution operator will be used to contract consecutive appearances of the same partition $I$, with $|I|=h-1$ in the right-hand side of \eqref{up_bnd_exp}. In particular, note that if $I\vdash \{1,\dots,h\}$, such that $|I|=h-1$, then
\begin{align*}
        \sfU^I_n(\bx,\by)=\sum_{k \geq 0} \bbE[\xi^2]^k \,\sum_{0:=n_0< n_1<\dots<n_k:=n} \, \,\sumtwo{\by_i \in (\Z^2)^h_I\, , 1\leq i \leq k-1,}{\by_0:=\bx, \, \by_k:=\by} \,\prod_{i=1}^k Q^{I;I}_{n_i-n_{i-1}}(\by_{i-1},\by_i) \, .
    \end{align*}
To be able to estimate the right-hand side of \eqref{up_bnd_exp} we will upper bound it by enlarging the domain of the temporal sum in the right-hand side of \eqref{up_bnd_exp} from $1\leq n_1 < \dots < n_{k+1} \leq 2N$ to $n_{i}-n_{i-1} \in \{1,\dots,2N\}$ for all $1 \leq i \leq k+1$. This enlargement of the domain of summation deconvolves the temporal sum in the right-hand side of \eqref{up_bnd_exp}.

On this account, we introduce the discrete Laplace transforms of the operators $\sfQ$ and $\sfU$,
\begin{align}
     & \sfQ_{ N,\lambda}^{I, J} (\by, \bz) := \sum_{n=1}^{2N} e^{-\lambda \frac{n}{N}}\,  Q^{I, J}_n(\by, \bz),   \,\qquad  \by, \bz \in (\Z^2)^h , 
     \label{LaplaceQ}\\
     & \sfU_{ N,\lambda}^{I}(\by, \bz) := \sum_{n=0}^{2N} e^{-\lambda \frac{n}{N}} \, \sfU^I_n (\by, \bz),  \qquad\quad\by, \bz \in (\Z^2)^h \, ,
     \label{LaplaceU}
\end{align}
for $\lambda \geq 0$. 
\begin{remark}\label{rem:positive_lambda}
In our case, it will be sufficient to work with $\lambda=0$, however, keeping a non zero $\lambda$
and tracking the dependence of the estimate on it (together with a closer tracking on the order $h$ of the moment)
would be necessary in order to extend our method so that to accommodate $h$ that grows with $N$.
\end{remark}
Let us define
\begin{align*}
    \sfP^{I;J}_{N,\bht}= \begin{dcases}
        \sfQ^{I;J}_{N,0} \, ,                 &  \text{ if   }  |J|<h-1  \\
        \sfQ^{I;J}_{ N,0} \, \sfU^J_{N,0}\, , & \text{ if } |J|=h-1 \, .
    \end{dcases}
\end{align*}
Note that the appearance of the operator $\sfU^J_{N,0}$ is necessarily preceded by a free evolution operator $\sfQ^{I;J}_{N,0}$, with $|J|=h-1$, see also Figure \ref{fig:4thmom}.
In view of \eqref{up_bnd_exp} and the discussion above we can now write
\begin{align*}
    \big|M^{\varphi, \psi}_{N,h}\big| \leq & \frac{c\norm{\psi}^h_{\infty}}{N^{h+1}}   \sum_{k\geq 1} \sum_{(I_1,\dots,I_k)  \in \mathcal{I}} \Big \langle \varphi_N^{\otimes h},  \sfP^{*,I_1}_{N,\bht} \, \sfP^{I_1,I_2}_{N,\bht} \cdots\,  \sfP_{N,\bht}^{I_k,*} \ind^{\otimes h}_{\sqrt{N} B}\Big   \rangle \prod_{i=1}^k \bbE\Big[|\xi|^{I_i}\Big]  \, ,
\end{align*}
where we recall the definition of the pairing $\langle \cdot\, , \cdot \rangle$ from \eqref{pairing} and note that the sum runs over partitions $I_1, \dots, I_k$ such that $I_j \neq I_{j+1}$ if $|I_j|=|I_{j+1}|=h-1$ for $1 \leq j \leq k-1$.

Because of the assumption of Theorem \ref{mom_est} on $\psi$ being merely a bounded function we will need to introduce weighted versions of the operators $\sfU^I_{N,\lambda}$, $\sfQ^{I;J}_{N,\lambda}$ and $\sfP^{I;J}_{N,\bht}$. In particular, if $w: \R^2 \to  \R$  is such that $\log w$ is Lipschitz continuous with Lipschitz constant denoted by $C_{w}>0$, 
and if we denote by $w_N(x)=w\big(\frac{x}{\sqrt{N}} \big )$, then we define for $\lambda \geq 0$,
\begin{align*}
     & \widehat{\sfQ}^{I;J}_{N,\lambda}(\bx,\by):= \frac{w_N^{\otimes h}( \bx)}{w^{\otimes h }_N(\by)} \, \sfQ^{I;J}_{N,\lambda}(\bx,\by)  \, ,\notag \\
     & \widehat{\sfU}^{I}_{N,\lambda}(\bx,\by):=\frac{w_N^{\otimes h}( \bx)}{w^{\otimes h }_N(\by)}\, \sfU^{I}_{N,\lambda}(\bx,\by) \, ,
\end{align*}
where we recall that $w_N^{\otimes h}(\bx)=w_N(x_1) \cdots w_N(x_h)$, if $\bx=(x_1, \dots, x_h)$. We modify accordingly the operator $\sfP^{I;J}_{N,\bht}$ into a new operator $\widehat{\sfP}^{I;J}_{N,\bht}$,
\begin{align} \label{P_IJ_def}
    \widehat{\sfP}^{I;J}_{N,\bht}= \begin{dcases}
        \widehat{\sfQ}^{I;J}_{N,0}                            & , \text{ if   }  |J|<h-1  \\
        \widehat{\sfQ}^{I;J}_{ N,0} \, \widehat{\sfU}^J_{N,0} & ,\text{ if } |J|=h-1 \, .
    \end{dcases}
\end{align}
Therefore, we can now write
\begin{equation} \label{exp_with_pairing}
    \begin{split}
    \big|M^{\varphi, \psi}_{N,h}\big| \leq & \frac{c\norm{\psi}^h_{\infty}}{N^{h+1}}   \sum_{k\geq 1} \sum_{(I_1,\dots,I_k)  \in \mathcal{I}} \Big \langle \frac{\varphi_N^{\otimes h}}{w_N^{\otimes h}},  \widehat{\sfP}^{*,I_1}_{N,\bht} \, \widehat{\sfP}^{I_1,I_2}_{N,\bht} \cdots\,  \widehat{\sfP}_{N,\bht}^{I_k,*} \ind^{\otimes h}_{\sqrt{N} B} w_N^{\otimes h}\Big   \rangle \prod_{i=1}^k \bbE\Big[|\xi|^{I_i}\Big]  \, \\
    \leq & \frac{c\norm{\psi}^h_{\infty}}{N^{h+1}}   \sum_{k\geq 1} \sum_{(I_1,\dots,I_k)  \in \mathcal{I}} \Big \langle \frac{\varphi_N^{\otimes h}}{w_N^{\otimes h}},  \widehat{\sfP}^{*,I_1}_{N,\bht} \, \widehat{\sfP}^{I_1,I_2}_{N,\bht} \cdots\,  \widehat{\sfP}_{N,\bht}^{I_k,*}  w_N^{\otimes h}\Big   \rangle \prod_{i=1}^k \bbE\Big[|\xi|^{I_i}\Big]
    \end{split}
\end{equation}
where we bounded the indicator function $\ind^{\otimes h}_{\sqrt{N}B}$ by $1$ to obtain the second inequality.
Passing to the operator norms (see \eqref{op_norm}) we estimate
\begin{align} \label{final_expansion}
    \big|M^{\varphi,\psi}_{N,h}\big | \leq & \frac{c\norm{\psi}^h_{\infty}}{N^{h+1}}  \sum_{k\geq 1} \sum_{(I_1,\dots,I_k)  \in \mathcal{I}}   \norm{\widehat{\sfP}^{*,I_1}_{N,\bht} \frac{\varphi_N^{\otimes h}}{w^{\otimes h }_N}}_{\ell^p} \, \prod_{i=2}^k\,  \norm{\widehat{\sfP}_{N,\bht}^{I_{i-1},I_i}}_{\ell^q \to \ell^q} \, \norm{\widehat{\sfP}_{N,\bht}^{I_k,*}  w_N^{\otimes h}}_{\ell^q } \prod_{i=1}^k  \bbE\Big[|\xi|^{I_i}\Big]  \, .
\end{align}
This is the key expansion we will use for the Proof of Theorem \ref{mom_est}.

\subsection{Integral inequalities for the operators \texorpdfstring{$\widehat{\sfQ}^{I;J}_{N,0}$}{} and \texorpdfstring{$\widehat{\sfU}^{I}_{N,0}$} .} \label{integral_inequalities}
At this point, we will prove the key estimates
 about the operators $\widehat{\sfQ}^{I;J}_{N,0},\widehat{\sfU}^{I}_{N,0}$ that we will need along the way.
In what follows we shall use the letter $C$ to denote constants that may depend only on $h,\bht$ and $w$ but not on $p$ and $q$. We will also use the letter $c$ to denote absolute constants, i.e. constants that do not depend on $h,\bht,w$ or $p,q$. Their value may change from line to line.

We start this subsection by stating a lemma on the operator $\sfQ_{N,\lambda}(\bx,\by):=\sum_{n=1}^{2N} e^{- \tfrac{\lambda n}{N}} Q_n(\bx,\by)$ from \cite{CSZ21} which we will need in the sequel.
\begin{lemma}[\cite{CSZ21}] \label{greens_fun}
    Let $N \geq 1$, $h \geq 2$ and $\bx,\by \in (\Z^2)^h$. Then, there exists a constant $C \in (0, \infty )$ such that
    uniformly in $N$, $\bx,\by$ and $\lambda \geq 0$,
    \vspace{0.2cm}
    \begin{align*}
        \sfQ_{N,\lambda}(\bx,\by) \leq  \begin{dcases}
            \, \,\frac{C}{\big(1+|\bx-\by|^2\big)^{h-1}}                        & \text{ for all  } \bx,\by \in (\Z^2)^h  \, , \\
            \,\, 	\frac{C}{N^{h-1}}\,	\exp \bigg(\frac{-|\bx-\by|^2}{C\, N}\bigg) & \text{ if }  |\bx-\by|> \sqrt{N} \, .
        \end{dcases}
    \end{align*}
\end{lemma}

The next proposition contains the central estimate. It is on the operator norm of
operator $\widehat{\sfQ}^{I;J}_{N,0}$, as an operator from an $\ell^q\to \ell^q$,
containing the explicit dependence on the parameters $p,q$.
\begin{proposition} \label{Q_IJ_norm}
    Let $p,q \in (1,\infty)$ such that $\tfrac{1}{p}+\tfrac{1}{q}=1$. There exists a constant $C=C(h,w) \in (0,\infty)$, independent of $p$ and $q$, such that for all $I,J \vdash \{1,\dots,h\}$ with $1 \leq |I|,|J|\leq h-1$ and $I \neq J$ when $|I|=|J|=h-1$,
    \begin{align}
        \norm{\widehat{\sfQ}^{I;J}_{N,0}}_{\ell^q \to \ell^q} \leq C \, p \, q \, .   \label{Q_norm}
    \end{align}
\end{proposition}

\begin{proof}
    Let $I,J \vdash \{1,\dots,h\}$ with $1 \leq |I|,|J| \leq h-1$ and $I\neq J$ when $|I|=|J|=h-1$ and consider $f \in \ell^{p}\big((\Z^2)^h_{I}\big)$, $g \in \ell^{q}\big((\Z^2)^h_{J}\big)$. In view of \eqref{op_norm}, in order to prove \eqref{Q_norm}, we need to prove that there exists a constant $C\in (0,\infty)$ such that
    \begin{align} \label{Q_ineq_prop}
        \sum_{\bx \in (\Z^2)^h_{I}, \by \in (\Z^2)^h_{J}} f(\bx)\sfQ^{I;J}_{N,0}(\bx,\by)\frac{w^{\otimes h}_N( \bx)}{w^{\otimes h}_N (\by)}g(\by) \leq C \, p \, q \,  \norm{f}_{\ell^p} \norm{g}_{\ell^q} \, .
    \end{align}
    Let
    \begin{align} \label{E_N_def}
        E_N:=\Big\{(\bx,\by) \in (\Z^2)^h_{I}\times (\Z^2)^h_{J}: |\bx- \by |\leq C_0 \sqrt{N} \Big\} \, .
    \end{align}
    for some $C_0>0$ to be determined.  By the second inequality in Lemma \ref{greens_fun} and the Lipschitz condition on $\log w$, we can choose $C_0$ large enough so that for all $(\bx,\by) \in E_N^c$ we have
    \begin{align*}
        \sfQ^{I;J}_{N,0}(\bx,\by)\frac{w^{\otimes h}_N( \bx)}{w^{\otimes h}_N( \by)} \leq \frac{C}{N^{h-1}} \exp \big(-\tfrac{|\bx- \by|}{\sqrt{N}} \big)  \, .
    \end{align*}
    Therefore, on $E_N^c$ we have that
    \begin{align*}
        \sum_{(\bx,\by) \in E_N^c } f(\bx)\sfQ^{I;J}_{N,0}(\bx,\by)\frac{w^{\otimes h}_N( \bx)}{w^{\otimes h}_N(\by)}g(\by)  \leq \frac{C}{N^{h-1}} \sum_{(\bx,\by) \in E_N^c} f(\bx) \exp \big(-\tfrac{|\bx- \by|}{\sqrt{N}} \big) g(\by) \,
    \end{align*}and by H\"{o}lder's inequality,
    \begin{align}    \label{E_N^c}
             & \frac{1 }{N^{h-1}} \sum_{(\bx,\by) \in E_N^c} f(\bx) \exp \big(-\tfrac{|\bx- \by|}{\sqrt{N}} \big) g(\by) \notag                                                                                                                                                                                             \\
        \leq & \, \frac{1}{N^{h-1}} \,  \Bigg( \sum_{\bx \in (\Z^2)^h_{I}, \by \in (\Z^2)^h_{J}} |f(\bx)|^p \exp \Big(-\tfrac{|\bx- \by|}{\sqrt{N}} \Big) \Bigg)^{\frac{1}{p}} \Bigg( \sum_{\bx \in (\Z^2)^h_{I}, \by \in (\Z^2)^h_{J}} |g(\by)|^q \exp \Big(-\tfrac{|\bx- \by|}{\sqrt{N}} \Big) \Bigg)^{\frac{1}{q}}\notag \\
        \leq & \, C  \, N^{\textstyle \frac{|J|}{p}+\frac{|I|}{q}-(h-1)} \norm{f}_{\ell^p} \norm{g}_{\ell^q}                                                                                                                                       \notag                                                                   \\
        \leq & \,  C \,  \norm{f}_{\ell^p} \norm{g}_{\ell^q} \, ,
    \end{align}
    where the inequality in the last line of \eqref{E_N^c} follows by the assumption $|I|,|J|\leq h-1$.
    Thus,
    \begin{align*}
        \sum_{(\bx,\by) \in E_N^c } f(\bx)\sfQ^{I;J}_{N,0}(\bx,\by)\frac{w^{\otimes h}_N( \bx)}{w^{\otimes h}_N(\by)}g(\by) \leq  C \,   \norm{f}_{\ell^p} \norm{g}_{\ell^q} \, ,
    \end{align*}
    for a constant $C \in (0,\infty)$.
    
    Let us now estimate the sum over $(\bx,\by)\in E_N$.
    Recalling that $\log w$ is Lipschitz with Lipschitz constant $C_w$ and \eqref{E_N_def}, we get that
    \begin{align*}
        \sum_{(\bx,\by)\in E_N}  f(\bx)\sfQ^{I;J}_{N,0}(\bx,\by)\frac{w^{\otimes h}_N(\bx)}{w^{\otimes h}_N(\by)}g(\by) \leq
        e^{C_w \, C_0}  \sum_{(\bx,\by)\in E_N }f(\bx)\sfQ^{I;J}_{N,0}(\bx,\by)g(\by) \, .
    \end{align*}
    Therefore, using the first inequality of Lemma \ref{greens_fun}, the key step is to show that there exists a constant $C \in (0,\infty)$ that may depend on $h$ and $w$ but not on $p$ and $q$, such that
    \begin{align} \label{Q_ineq}
        \sum_{\bx \in (\Z^2)^h_I, \, \by \in  (\Z^2)^h_J} \frac{f(\bx)g(\by)}{\Big(1+\sum_{i=1}^h
            |x_i-y_i|^2\Big)^{h-1}}\leq C \, p \, q \norm{f}_{\ell^p} \norm{g}_{\ell^q} \, .
    \end{align}
    By assumption there exist $1\leq k,\ell \leq h $ such that $k \stackrel{I}{\sim} \ell$ and $1\leq m,n\leq h$ such that $m \stackrel{J}{\sim} n$. Since we have assumed that $I\neq J$ when $|I|=|J|=h-1$, we may assume without loss of generality that $m \neq k,\ell$. Let $a \in \big(0,\min \{p^{-1}, q^{-1}\}\big)$ to be determined later. By multiplying and dividing by $ \frac{1+|x_m-x_n|^{2a}}{1+|y_k-y_{\ell}|^{2a}}$ and using H\"{o}lder's inequality, the left-hand side of \eqref{Q_ineq} is upper bounded  by
    \begin{align} \label{Q_ineq_H}
        \Bigg(\sum_{\bx \in (\Z^2)^h_I, \, \by \in  (\Z^2)^h_J} & \frac{|f(\bx)|^p}{\Big(1+\sum_{i=1}^h |x_i-y_i|^2\Big)^{h-1}}  \cdot \frac{\big(1+|x_m-x_n|^{2a}\big)^p}{\big(1+|y_k-y_{\ell}|^{2a}\big)^p}\Bigg)^{\frac{1}{p}}  \notag                                                            \\
                                                                & \times \Bigg(\sum_{\bx \in (\Z^2)^h_I, \, \by \in  (\Z^2)^h_J} \frac{|g(\by)|^q}{\Big(1+\sum_{i=1}^h |x_i-y_i|^2\Big)^{h-1}} \cdot \frac{\big(1+|y_k-y_{\ell}|^{2a}\big)^q}{\big(1+|x_m-x_n|^{2a}\big)^q}\Bigg)^{\frac{1}{q}} \, .
    \end{align}
    By symmetry, it is enough to bound one of the two factors in \eqref{Q_ineq_H}. By triangle inequality and the fact that $m \stackrel{J}{\sim} n$, which means that $y_m=y_n$, we have
    \begin{align*}
        |x_m-y_m|^2+|x_n-y_n|^2 \geq
        \frac{|x_m-x_n|^2+|x_n-y_n|^2}{4} \, .
    \end{align*}
    Therefore,
    \begin{align}\label{f_p_norm}
             & \Bigg(\sum_{\bx \in (\Z^2)^h_I, \, \by \in  (\Z^2)^h_J} \frac{|f(\bx)|^p}{\Big(1+\sum_{i=1}^h |x_i-y_j|^2 \Big)^{h-1}}  \cdot \frac{\big(1+|x_m-x_n|^{2a}\big)^p}{\big(1+|y_k-y_{\ell}|^{2a}\big)^p}\Bigg)^{\frac{1}{p}} \\
        \leq & \,  4^{\frac{h-1}{p}} 	\Bigg(\sum_{\bx \in (\Z^2)^h_I} |f(\bx)|^p\,   (1+ |x_m-x_n|^{2a})^p \,\,   \notag                                                                                                                 \\ &   \qquad \qquad \qquad \times\sum_{\by \in (\Z^2)^h_J} \frac{1}{\Big(1+ |x_m-x_n|^2 + \sum_{i \neq m}|x_i-y_i|^2 \Big)^{h-1} \big(1+|y_k-y_{\ell}|^{2a}\big)^p}  \Bigg)^{\frac{1}{p}}  \notag  \, .
    \end{align}
    By using \eqref{s1} of Lemma \ref{sums_bnds} and summing successively the $y_i$ variables for $i \neq k,\ell$ we obtain that
    \begin{align*}
             & \sum_{\by \in (\Z^2)^h_J} \frac{1}{\Big(1+ |x_m-x_n|^2 + \sum_{i \neq m} |x_i-y_i|^2 \Big)^{h-1} (1+|y_k-y_{\ell}|^{2a})^p}                                           \\
        \leq & \,  c^{|J|-2} \sum_{y_k,\, y_{\ell} \in \Z^2} \frac{1}{\Big(1+|x_m-x_n|^2+|y_k-x_{k}|^2+|y_{\ell}-x_{\ell}|^2 \Big)^{h+1-|J|} \big(1+|y_k-y_{\ell}|^{2a}\big)^p} \, .
    \end{align*}
    We make a change of variables $w_1=y_k-y_{\ell}$ and $w_2=y_k+y_{\ell}-2x_k$ and observe that $\frac{w^2_1+w^2_2}{2}=|y_k-x_k|^2+|y_{\ell}-x_{\ell}|^2$, where we used that $k \stackrel{I}{\sim} \ell$ thus $x_k=x_{\ell}$. Therefore, we have
    \begin{align*}
             & c^{|J|-2} \sum_{y_k,\, y_{\ell} \in \Z^2} \frac{1}{\Big(1+|x_m-x_n|^2+|y_k-x_k|^2+|y_{\ell}-x_{\ell}|^2 \Big)^{h+1-|J|} \big(1+|y_k-y_{\ell}|^{2a}\big)^p} \\
        \leq & \, 2^{h+1-|J|}\, c^{|J|-2}\sum_{w_1,\, w_2 \in \Z^2} \frac{1}{\Big(1+|x_m-x_n|^2+|w_1|^2+|w_2|^2 \Big)^{h+1-|J|} \big(1+|w_1|^{2a}\big)^p} \, .
    \end{align*}
    By summing $w_2$ and using \eqref{s1} of Lemma \ref{sums_bnds} we have,
    \begin{align*}
             & 2^{h+1-|J|}\, c^{|J|-2} \sum_{w_1,\, w_2 \in \Z^2} \frac{1}{\Big(1+|x_m-x_n|^2+|w_1|^2+|w_2|^2 \Big)^{h+1-|J|} \big(1+|w_1|^{2a}\big)^p} \\
        \leq & \, 2^{h+1-|J|}\, c^{|J|-1}\, \, \sum_{w_1 \in \Z^2} \frac{1}{\Big(1+|x_m-x_n|^2+|w_1|^2 \Big)^{h-|J|} \big(1+|w_1|^{2a}\big)^p}
    \end{align*}
    By \eqref{s2} of Lemma \ref{sums_bnds} we have that
    \begin{align*}
             & 2^{h+1-|J|}\, c^{|J|-1} \, \sum_{w_1 \in \Z^2} \frac{1}{\Big(1+|x_m-x_n|^2+|w_1|^2 \Big)^{h-|J|} \big(1+|w_1|^{2a}\big)^p} \\
        \leq & \,  2^{h+1-|J|}\,c^{|J|} \, \frac{1 }{ap(1-ap)} \frac{1}{\big(1+|x_m-x_n|^2\big)^{ap+h-1-|J|}}                             \\
        \leq & \, 2^{h+1-|J|}\,c^{|J|} \frac{1}{ap(1-ap)} \frac{1}{\big(1+|x_m-x_n|^2\big)^{ap}} \, ,
    \end{align*}
    where in the last inequality we used that $|J|\leq h-1$ by assumption. Therefore, the right-hand side of \eqref{f_p_norm} is bounded by
    \begin{align} \label{f_rhs}
          & \bigg( 4^{h-1} \, 2^{h+1}\, \Big(\frac{c}{2}\Big)^{|J|} \frac{1 }{ap(1-ap)}\, \bigg)^{\frac{1}{p}} \cdot	\Bigg(\sum_{\bx \in (\Z^2)^h_I} |f(\bx)|^p\,   \frac{(1+ |x_m-x_n|^{2a})^p}{\big(1+|x_m-x_n|^2\big)^{ap}}   \Bigg)^{\frac{1}{p}} \notag \\
        = & \bigg( 2^{3h-1}\, \Big(\frac{c}{2}\Big)^{|J|} \frac{1 }{ap(1-ap)}\, \bigg)^{\frac{1}{p}}\cdot	\Bigg(\sum_{\bx \in (\Z^2)^h_I} |f(\bx)|^p\,   \frac{(1+ |x_m-x_n|^{2a})^p}{\big(1+|x_m-x_n|^2\big)^{ap}}   \Bigg)^{\frac{1}{p}} \, .
    \end{align}
    Note furthermore, that
    \begin{align*}
        \frac{\big(1+|x_m-x_n|^{2a}\big)^p}{\big(1+|x_m-x_n|^2\big)^{ap}}\leq \frac{2^p \max\big\{ 1,|x_m-x_n|\big\}^{2ap}}{\big(1+|x_m-x_n|^2\big)^{ap}}\leq 2^p\, ,
    \end{align*}
    therefore,
    \begin{align*}
        \Bigg(\sum_{\bx \in (\Z^2)^h_I} |f(\bx)|^p\,   \frac{(1+ |x_m-x_n|^{2a})^p}{\big(1+|x_m-x_n|^2\big)^{ap}}   \Bigg)^{\frac{1}{p}} \leq 2 \Bigg(\sum_{\bx \in (\Z^2)^h_I} |f(\bx)|^p\,     \Bigg)^{\frac{1}{p}}=\,2  \norm{f}_{\ell^p} \, .
    \end{align*}
    Hence, setting
    \begin{align*}
        C^{J}_{p,h}:= 2\cdot  \Bigg( 2^{3h-1}\, \Big(\frac{c}{2}\Big)^{|J|} \frac{1 }{ap(1-ap)}\Bigg)^{\frac{1}{p}} \,
    \end{align*}
    and recalling \eqref{f_p_norm}, \eqref{f_rhs} we get that
    \begin{align}  \label{f_p_bnd}
        \Bigg(\sum_{\bx \in (\Z^2)^h_I, \, \by \in  (\Z^2)^h_J} \frac{|f(\bx)|^p}{\Big(1+\sum_{i=1}^h |x_i-y_i|^2\Big)^{h-1}}  \cdot \frac{\big(1+|x_m-x_n|^{2a}\big)^p}{\big(1+|y_k-y_{\ell}|^{2a}\big)^p}\Bigg)^{\frac{1}{p}}
        \leq     \,C^{J}_{p,h}\, \norm{f}_{\ell^p} \, .
    \end{align}
    By symmetry we also obtain that
    \begin{align} \label{g_q_bnd}
        \Bigg(\sum_{\bx \in (\Z^2)^h_I, \, \by \in  (\Z^2)^h_J} \frac{|g(\by)|^q}{\Big(1+\sum_{i=1}^h |x_i-y_i|^2\Big)^{h-1}}  \cdot \frac{\big(1+|y_k-y_{\ell}|^{2a}\big)^q}{\big(1+|x_m-x_n|^{2a}\big)^q}\Bigg)^{\frac{1}{q}} \leq  C^{I}_{q,h}\, \norm{g}_{\ell^q} \, ,
    \end{align}
    with
    \begin{align*}
        C^{I}_{q,h}:= 2\cdot  \Bigg( 2^{3h-1}\, \Big(\frac{c}{2}\Big)^{|I|} \frac{1 }{aq(1-aq)}\Bigg)^{\frac{1}{q}} \, .
    \end{align*}
    Consequently, recalling \eqref{Q_ineq} and using \eqref{f_p_bnd}, \eqref{g_q_bnd} we deduce that
    \begin{align*}
        \sum_{\bx \in (\Z^2)^h_I, \, \by \in  (\Z^2)^h_J} \frac{f(\bx) g(\by)}{\Big(1+\sum_{i=1}^h |x_i-y_i|^2\Big)^{h-1}}\leq \,C^{J}_{p,h}\, C^{I}_{q,h} \norm{f}_{\ell^p} \norm{g}_{\ell^q} \, .
    \end{align*}
    We optimise by choosing $a=(p\, q)^{-1}$ so as to obtain
    \begin{align*}
        C^J_{p,h}= 2\cdot  \Bigg( 2^{3h-1}\, \Big(\frac{c}{2}\Big)^{|J|}\, p\, q\,\Bigg)^{\frac{1}{p}}\, \text{ and } \,C^I_{q,h}= 2\cdot  \Bigg( 2^{3h-1}\, \Big(\frac{c}{2}\Big)^{|I|}\, p\, q\,\Bigg)^{\frac{1}{q}} \, ,
    \end{align*}
    which implies that
    \begin{align*}
        C^J_{p,h} \, C^I_{q,h}= 2^{3h+1} \, \Big(\frac{c}{2}\Big)^{\frac{|J|}{p}+\frac{|I|}{q}} \, p\, q \, .
    \end{align*}
    Noting that $\big(\frac{c}{2}\big)^{\frac{|J|}{p}+\frac{|I|}{q}}\leq \max\Big\{1,\big(\frac{c}{2}\big)^{h-1}\Big\}$, we deduce that there exists $C=C(h,w) \in (0,\infty)$ such that
    \begin{align*}
        \sum_{\bx \in (\Z^2)^h_I, \, \by \in  (\Z^2)^h_J}  \frac{f(\bx)g(\by)}{\Big(1+\sum_{i=1}^h |x_i-y_i|^2\Big)^{h-1}}\leq C \, p \, q \, \norm{f}_{\ell^p} \norm{g}_{\ell^q} \, ,
    \end{align*}
    which together with \eqref{E_N^c} imply \eqref{Q_ineq_prop}.
\end{proof}
The next proposition is the analogue of Proposition \ref{Q_IJ_norm} for the boundary operators.
\begin{proposition} \label{Q_left_norm}
    Let $p,q \in (1,\infty)$ such that $\frac{1}{p}+\frac{1}{q}=1$.
    There exists a constant $C=C(h,w) \in (0,\infty)$, independent of $p$ and $q$, such that for all $I\vdash \{1,\dots,h\}$ with $|I|\leq h-1$ and $g\in \ell^q(\Z^2)$,
    \begin{align*}
        \norm{\widehat{\sfQ}^{I;*}_{N,0} \, g^{\otimes h }}_{\ell^q}  \leq  \, C \,  p  \,  N^{\frac{1}{p}} \norm{g}^{h}_{\ell^q}\, .
    \end{align*}
\end{proposition}

\begin{proof}
    Let $I\vdash\{1,\dots,h\}$ with $|I|\leq h-1$. 
    In order to prove  Proposition \ref{Q_left_norm}, we need to show that
    \begin{align*}
        \sum_{\bx \in (\Z^2)^h_I, \, \by \in  (\Z^2)^h} f(\bx) \sfQ^{I;*}_{N,0}(\bx,\by) \frac{w_N^{\otimes h}( \bx)}{w_N^{\otimes h }(\by)} g^{\otimes h}(\by) \leq   C \, p\, N^{\frac{1}{p}}\norm{f}_{\ell^p} \norm{g}^h_{\ell^q} \, .
    \end{align*}
    for any $f \in \ell^p\big((\Z^2)^{|I|}\big)$. The proof of this Proposition is a modification of the proof of Proposition \ref{Q_IJ_norm}. Let
    \begin{align*}
        E_N:=\Big\{(\bx,\by) \in (\Z^2)^h_{I}\times (\Z^2)^h: |\bx- \by |\leq C_0 \sqrt{N} \Big\} \, .
    \end{align*}
    For $(\bx,\by) \in E_N^c$, following \eqref{E_N^c} we have
    \begin{align*}
        \sum_{(\bx,\by)\in E_N^c} f(\bx) \frac{w_N^{\otimes h}( \bx)}{w_N^{\otimes h }(\by)} \sfQ^{I;*}_{N,0}(\bx,\by) g^{\otimes h}(\by)
        \leq & \,  C\,   N^{\textstyle \frac{h}{p}+\frac{|I|}{q}-(h-1)} \norm{f}_{\ell^p} \norm{g}_{\ell^q}^h    \notag \\                                   \leq &\, C\, N^{\frac{1}{p}} \norm{f}_{\ell^p} \norm{g}^h_{\ell^q} \, ,
    \end{align*}
    since $|I|\leq h-1$. Therefore, in light of the first inequality of Lemma \ref{greens_fun}, it remains to show that
    \begin{align} \label{Qpq_d}
        \sum_{(\bx , \by) \in  E_N} \frac{f(\bx)g^{\otimes h}(\by)}{\Big(1+\sum_{i=1}^h |x_i-y_j|^2\Big)^{h-1}}\leq C \, p \, N^{\frac{1}{p}}\, \norm{f}_{\ell^p} \norm{g}^h_{\ell^q} \, .
    \end{align}
    We can assume without loss of generality that $1 \stackrel{I}{\sim} 2$, that is $x_1=x_2$. We  multiply and divide by the factor $\Big(\log \Big(1+\frac{C^2_0 N}{1+|y_1-y_2|^2} \Big)\Big)^{\frac{1}{q}}$ in \eqref{Qpq_d} and apply H\"{o}lder's inequality, namely
    \begin{align} \label{Qpq_log}
             & \sum_{(\bx , \by) \in  E_N} \frac{f(\bx)g^{\otimes h}(\by)}{\Big(1+\sum_{i=1}^h |x_i-y_i|^2\Big)^{h-1}} \notag                                                                                                     \\
        \leq & \Bigg(\sum_{(\bx , \by) \in  E_N} \frac{|f(\bx)|^p\, \Big(\log \Big(1+\frac{ C^2_0 N}{1+|y_1-y_2|^2} \Big)\Big)^{\frac{p}{q}}}{\Big(1+\sum_{i=1}^h |x_i-y_i|^2\Big)^{h-1}}\Bigg)^{\frac{1}{p}} \,\, \,\notag \\ & \qquad \qquad \qquad \times \Bigg(\sum_{(\bx , \by) \in  E_N} \frac{|g^{\otimes h}(\by)|^q}{\Big(1+\sum_{i=1}^h |x_i-y_i|^2\Big)^{h-1}\log \Big(1+\frac{ C^2_0 N}{1+|y_1-y_2|^2} \Big)}\Bigg)^{\frac{1}{q}} \, .
    \end{align}
    By triangle inequality and using that $x_1=x_2$ we have that
    \begin{align*}
        \displaystyle |x_1-y_1|^2+|x_2-y_2|^2 \geq \frac{|y_1-y_2|^2+|x_2-y_2|^2}{4}\, ,
    \end{align*}
    therefore
    \begin{align} \label{log_after_triang}
             & \Bigg(\sum_{(\bx , \by) \in  E_N} \frac{|g^{\otimes h}(\by)|^q}{\Big(1+\sum_{i=1}^h |x_i-y_i|^2\Big)^{h-1}\log \Big(1+\frac{ C^2_0 N}{1+|y_1-y_2|^2} \Big)}\Bigg)^{\frac{1}{q}}                              \notag    \\
        \leq & \, 4^{\frac{h-1}{q}} \Bigg(\sum_{(\bx , \by) \in  E_N} \frac{|g^{\otimes h}(\by)|^q}{\Big(1+|y_1-y_2|^2+\sum_{i=2}^{h} |x_i-y_i|^2\Big)^{h-1}\log \Big(1+\frac{C^2_0 N}{1+|y_1-y_2|^2} \Big)}\Bigg)^{\frac{1}{q}} \, .
    \end{align}
    We sum the $x_i$ variables for $i > 2$ successively, so that by inequality \eqref{s1} of Lemma \ref{sums_bnds},
    \begin{align} \label{log_sum_succ}
             & \sum_{\bx \in (\Z^2)^{h}_I:\,  (\bx,\by) \in E_N}\frac{1}{\Big(1+|y_1-y_2|^2+\sum_{i=2}^{h} |x_i-y_i|^2\Big)^{h-1}\log \Big(1+\frac{ C^2_0 N}{1+|y_1-y_2|^2} \Big)}                \notag \\
        \leq & \, c^{|I|-1} \frac{1}{\log \Big(1+\frac{  C^2_0 N}{1+|y_1-y_2|^2} \Big)} \sumtwo{x_2 \in \Z^2}{|x_2-y_2|\leq C_0 \sqrt{N}} \frac{1}{\big(1+ |y_1-y_2|^2+ |x_2-y_2|^2\big)^{h-|I|}} \, .
    \end{align}
    We also note that since $|I|\leq h-1$,
    \begin{align} \label{log_sum_2}
        \sumtwo{x_2 \in \Z^2}{|x_2-y_2|\leq  C_0 \sqrt{N}} \frac{1}{\big(1+ |y_1-y_2|^2+ |x_2-y_2|^2\big)^{h-|I|}} & \leq \sumtwo{x_2 \in \Z^2}{|x_2-y_2| \leq  C_0 \sqrt{N}} \frac{1}{1+ |y_1-y_2|^2+ |x_2-y_2|^2} \notag \\ & \leq  \,  c \log \bigg(1+\frac{C_0^2 N}{1+|y_1-y_2|^2} \bigg) \, ,
    \end{align}
    where the last inequality in \eqref{log_sum_2} follows from inequality \eqref{s3} of Lemma \ref{sums_bnds_2}.
    Thus, taking into account \eqref{log_sum_succ} and \eqref{log_sum_2} we deduce that
    \begin{align*}
        \sum_{\bx \in (\Z^2)^{h}_I:\,  (\bx,\by) \in E_N}\frac{1}{\Big(1+|y_1-y_2|^2+\sum_{i=2}^{h} |x_i-y_i|^2\Big)^{h-1}\log \Big(1+\frac{ C^2_0 N}{1+|y_1-y_2|^2} \Big)} \leq c^{|I|} \leq c^{h-1} \, ,
    \end{align*}
    since $|I|\leq h-1$. By \eqref{log_after_triang} we obtain that
    \begin{align} \label{g_q_log_bound}
        \Bigg(\sum_{(\bx , \by) \in  E_N} \frac{|g^{\otimes h}(\by)|^q}{\Big(1+\sum_{i=1}^h |x_i-y_i|^2\Big)^{h-1}\log \Big(1+\frac{ C^2_0 N}{1+|y_1-y_2|^2} \Big)}\Bigg)^{\frac{1}{q}}
         & \leq  \, (4c  )^{\frac{h-1}{q}}\Bigg(\sum_{ \by \in  (\Z^2)^h} |g^{\otimes h}(\by)|^q\Bigg)^{\frac{1}{q}} \notag \\ & =  (4c  )^{\frac{h-1}{q}} \norm{g}^h_{\ell^q}  \, .
    \end{align}
    On the other hand, for the first term in \eqref{Qpq_log}, using that $x_1=x_2$, by \eqref{s1} of Lemma \ref{sums_bnds}, we have that
    \begin{align} \label{log_f_1}
            \sum_{\by \in (\Z^2)^h} \frac{\Big(\log \Big(1+\frac{ C^2_0 N}{1+|y_1-y_2|^2} \Big)\Big)^{\frac{p}{q}}}{\Big( 1+\sum_{i=1}^h |x_i-y_i|^2\Big)^{h-1}} \leq c^{h-2} \sumtwo{y_1,y_2 \in \Z^2}{|y_1-x_1|,|y_2-x_1| \leq C_0 \sqrt{N}} \frac{\Big(\log \Big(1+\frac{ C^2_0 N}{1+|y_1-y_2|^2} \Big)\Big)^{\frac{p}{q}}}{\Big( 1+|x_1-y_1|^2+|x_1-y_2|^2 \Big)} \, .
        \end{align}

    We make the change of variables $w_1:=y_1-y_2$ and $w_2:=y_1+y_2-2x_1$, so that $|w_1|,|w_2|\leq 2 C_0 \sqrt{N}$ and $|w_1|^2+|w_2|^2=2 |y_1-x_1|^2+2|y_2-x_1|^2$. Note that then,
    \begin{align*}
            c^{h-2} \sumtwo{y_1,y_2 \in \Z^2}{|y_1-x_1|,|y_2-x_1| \leq C_0 \sqrt{N}} \frac{\Big(\log \Big(1+\frac{C^2_0 N}{1+|y_1-y_2|^2} \Big)\Big)^{\frac{p}{q}}}{\Big( 1+|x_1-y_1|^2+|x_1-y_2|^2 \Big)}
            \leq  \,2 c^{h-2} \sumtwo{w_1,w_2 \in \Z^2}{|w_1|,|w_2| \leq 2\, C_0 \sqrt{N}} \frac{\Big(\log \Big(1+\frac{ C^2_0 N}{1+|w_1|^2} \Big)\Big)^{\frac{p}{q}}}{ 1+|w_1|^2+|w_2|^2 } \, .
        \end{align*} 
    Next, we sum over $w_2$ and use inequality \eqref{s3} of Lemma \ref{sums_bnds_2} to obtain
     \begin{align} \label{log_f_2}
            \, 2 c^{h-2} \sumtwo{w_1,w_2 \in \Z^2}{|w_1|,|w_2| \leq 2\, C_0 \sqrt{N}} \frac{\Big(\log \Big(1+\frac{ C^2_0 N}{1+|w_1|^2} \Big)\Big)^{\frac{p}{q}}}{ 1+|w_1|^2+|w_2|^2 }
            \leq  \,  2 c^{h-1} \sumtwo{w_1\in \Z^2}{|w_1| \leq 2\, C_0 \sqrt{N}} \Big(\log \Big(1+\frac{ C^2_0 N}{1+|w_1|^2}\Big) \Big)^{\frac{p}{q}+1} \, .
        \end{align} 
    By \eqref{s4} of Lemma \ref{sums_bnds_2} and noting that $\tfrac{p}{q}+1=p$ we have
    \begin{align} \label{log_f_3}
        \sumtwo{w_1 \in \Z^2}{|w_1|\leq 2C_0 \sqrt{N}} \Big(\log\Big(1+ \frac{C^2_0 N}{1+|w_1|^2}\Big)\Big)^p  \leq c \,C_0^2\,  N \, p^p \, .
    \end{align}
    Therefore, by \eqref{log_f_1}, \eqref{log_f_2} and \eqref{log_f_3} we have that
    \begin{align} \label{f_p_log_bound}
        \Bigg(\sum_{(\bx , \by) \in  E_N} \frac{|f(\bx)|^p\, \Big(\log \Big(1+\frac{ C^2_0 N}{1+|y_1-y_2|^2} \Big)\Big)^{\frac{p}{q}}}{\Big(1+\sum_{i=1}^h |x_i-y_i|^2\Big)^{h-1}}\Bigg)^{\frac{1}{p}} \leq  & \Big(2\,  c^{h}\, C_0^2\Big)^{\frac{1}{p}}\,N^{\frac{1}{p}}\,    p \, \Bigg(\sum_{\bx  \in (\Z^2)^h_I} |f(\bx)|^p \Bigg)^{\frac{1}{p}}\notag \\ \leq & \Big( 2\, c^{h}\, C_0^2\Big)^{\frac{1}{p}}\,N^{\frac{1}{p}}\,    p \, \norm{f}_{\ell^p} \, .
    \end{align}
    Taking into account \eqref{g_q_log_bound}, \eqref{f_p_log_bound} and \eqref{Qpq_log}  we obtain that there exists $C=C(h,w) \in (0,\infty)$ such that
    \begin{align*}
        \sum_{(\bx,\by) \in E_N } \frac{f(\bx)g^{\otimes h}(\by)}{\Big(1+\sum_{i=1}^h |x_i-y_j|^2\Big)^{h-1}}\leq C \, p \, N^{\frac{1}{p}}\, \norm{f}_{\ell^p} \norm{g}^h_{\ell^q} \, ,
    \end{align*}
    which concludes the proof of \eqref{Qpq_d} and thus, the proof of Proposition \ref{Q_left_norm}.
\end{proof}

\begin{proposition} \label{U_norm}
    Let $p,q \in (1,\infty)$ such that $\frac{1}{p}+\frac{1}{q}=1$.	There exists a constant $C=C(h,\bht
        ,w) \in (0,\infty)$, independent of $p$ and $q$, such that for all $I\vdash\{1,\dots,h\}$ with $|I|=h-1$,
    \begin{align*}
        \norm{\widehat{\sfU}_{N,0
            }^I}_{\ell^q \to \ell^q} \leq C\, .
    \end{align*}
\end{proposition}

\begin{proof}
    Using \eqref{op_norm} it suffices to prove that if $f \in \ell^p\big( (\Z^2)^h_I \big)$, $g \in \ell^q \big((\Z^2)^h_I\big)$, then we have
    \begin{align*}
        \sum_{\bx, \by  \in (\Z^2)^h_I} f(\bx)\sfU^I_{N,0
        }(\bx,\by) \frac{w^{\otimes h }_N(\bx)}{w_N^{\otimes h }(\by )} g(\by) \leq C \norm{f}_{\ell^p} \norm{g}_{\ell^q} \, .
    \end{align*}
    By the Lipschitz condition on $\log w$ we first have
    \begin{align*}
        \sum_{\bx, \by  \in (\Z^2)^h_I} f(\bx)\sfU^I_{N,0
        }(\bx,\by) \frac{w^{\otimes h }_N(\bx)}{w_N^{\otimes h }(\by )} g(\by) \leq  \sum_{\bx, \by  \in (\Z^2)^h_I} f(\bx)\sfU^I_{N,0
        }(\bx,\by)e^{C_w \frac{|\bx-\by|}{\sqrt{N}}} g(\by) \, ,
    \end{align*}
    which by H\"{o}lder's inequality is bounded by
    \begin{align*}
        \Bigg(\sum_{\bx, \by  \in (\Z^2)^h_I} |f(\bx)|^p \,  \sfU^I_{N,0
        }(\bx,\by) e^{C_w \frac{|\bx-\by|}{\sqrt{N}}} \Bigg)^\frac{1}{p} \cdot \Bigg(\sum_{\bx, \by  \in (\Z^2)^h_I} |g(\by)|^q \,  \sfU^I_{N,0
        }(\bx,\by) e^{C_w \frac{|\bx-\by|}{\sqrt{N}}} \Bigg)^\frac{1}{q}  \, .
    \end{align*}
    Therefore, in order to conclude the proof of \eqref{U_norm} it suffices to prove that there exists a constant $C$ such that uniformly in $\bx \in (\Z^2)^h_I$,
    \begin{align} \label{U_eq_1}
        \sum_{\by \in (\Z^2)^h_I} \sfU^I_{N,0
        }(\bx,\by) e^{C_w \frac{|\bx-\by|}{\sqrt{N}}} \leq C \, .
    \end{align}

    Recall from \eqref{replica_op} that if $I$ is of the form $I=\{k,\ell \} \sqcup \bigsqcup_{j\neq k, \ell} \{j\}$ then for $\bx,\by \in (\Z^2)^h_I$ the operator $\sfU^I_{N,0
        }(\bx,\by)$ is defined as
    \begin{align*}
        \sfU^I_{N,0
        }(\bx,\by)=\sum_{n=0}^{2N} \sfU^I_n(\bx,\by) = \ind_{\{\bx,\by\sim I\}} \cdot \sum_{n=0}^{2N}    U^{\beta_N}_{N}(n,y_k-x_k) \cdot \prod_{i\neq k,\ell} q_n(y_i-x_i) \, .
    \end{align*}
    Therefore, in view of \eqref{U_eq_1}, we shall prove that uniformly in $0\leq n\leq 2N$,
    \begin{align} \label{U_eq_2}
        \sum_{z \in \Z^2} U^{\beta_N}_{N} (n,z) e^{C_w\frac{|z|}{\sqrt{N}}}\leq C\,  U^{\beta_N}_{N}(n)
    \end{align}
    and
    \begin{align} \label{U_eq_3}
        \sum_{z \in \Z^2} q_n(z) e^{C_w\frac{|z|}{\sqrt{N}}} \leq C\, q_n(z)\, .
    \end{align}
    Inequality \eqref{U_eq_3} follows easily by the local CLT, see \cite{LL10} and Gaussian concentration.
    For the sake of the presentation, we will prove \eqref{U_eq_2} for $0\leq n \leq N$, that is,
    \begin{align} \label{U_eq_2_sub}
        \sum_{z \in \Z^2} U^{\beta_N}_{N} (n,z) e^{C_w\frac{|z|}{\sqrt{N}}}\leq C\,  U^{\beta_N}_{N}(n) \, , \qquad \forall \, \, 0\leq n \leq N \, .
    \end{align}
    Note that, by \eqref{Uneqvar} we have,
    \begin{align} \label{var_bnd}
        \sum_{n=0}^N U^{\beta_N}_{N}(n)
        \leq  \bbE\Big[(Z_{N+1}^{\beta_N})^2\Big] \leq \frac{C}{1-\bht^2}
        \, .
    \end{align}
    Moreover, following the renewal framework we introduced in Section \ref{aux_tools}, we have
    \begin{align}
        \label{cond_times}
        \sum_{z \in \Z^2}U^{\beta_N}_N (n,z) e^{C_w\frac{|z|}{\sqrt{N}}} = & \sum_{k \geq 0} {\bht}
        ^{2k} \,  \E\bigg[ e^{C_w\frac{|S^{(N)}_k|}{\sqrt{N}}};\, \tau_k^{(N)}=n\bigg]              \\
        =           \notag                                                 & \sum_{k \geq 0} {\bht}
        ^{2k} \sum_{n_1+\dots+n_k=n}\,  \E\bigg[ e^{C_w\frac{|S^{(N)}_k|}{\sqrt{N}}}\, \Big| \,  \sft^{(N)}_i=n_i \, , 1 \leq i \leq k  \bigg] \prod_{i=1}^k \P\big(\sft_i^{(N)}=n_i\big) \, .
    \end{align}
    Therefore, in order to establish \eqref{U_eq_2_sub} it suffices to prove that there exists $C \in (0,\infty)$, such that for all $k \geq 1$,
    \begin{align} \label{exp_mom}
        \E\bigg[ e^{C_w\frac{|S^{(N)}_k|}{\sqrt{N}}}\, \Big| \,  \sft^{(N)}_i=n_i \, , 1 \leq i \leq k  \bigg] \leq C \, .
    \end{align}
    We note that when we condition on the times $\big(\sft^{(N)}_i\big)_{1\leq i \leq k}$, the space increments $\big(\sfx^{(N)}_i\big)_{1\leq i \leq k}$ are independent with distribution
    \begin{align*}
        \P\big(\sfx^{(N)}_1=x\, \big |\, \sft_1^{(N)}=n_1 \big)=\frac{q_{n_1}^2(x)}{q_{2n_1}(0)} \,  \ind_{\{n_1\leq N\} }\, .
    \end{align*}
    Let $\lambda \geq 0$ and $(\xi_{i})_{1\leq i \leq k}$ independent random variables such that 
    $\xi_i \stackrel{\text{law}}{=} \sfx_i^{(N)} \big|\, \sft^{(N)}_i=n_i$. We will show that
    \begin{align*}
        \E\Big[ e^{\lambda |\sum_{i=1}^k{\xi_i}|}\, \Big] \leq 2e^{4c \lambda^2 n} \, ,
    \end{align*}
    for some $c>0$. Therefore, taking $\lambda=\frac{C_w}{\sqrt{N}}$  will lead to  \eqref{exp_mom}. To this end, for each $1\leq i \leq k$, let $\xi_{i,1},\xi_{i,2} \in \Z$ be the two components of $\xi_i \in \Z^2$ .Then we can find $c>0$ such that
    \begin{align*}
        \E\big[e^{ \pm \,\lambda \, \xi_{i,j}}\big]\leq e^{c\lambda^2 n_i}
    \end{align*}
    for $j=1,2$, since by the local CLT we have
    \begin{align*}
        \P(\xi_{i}=x)=\frac{q_{n_i}^2(x)}{q_{2n_i}(0)}\leq \Big( \frac{\sup_{x \in \Z^2}q_{n_i}(x)}{q_{2n_i}(0)}\Big)\,q_{n_i}(x) \leq C' 	q_{n_i}(x) \,
    \end{align*}
    and $q_{n_i}(x)=2(g_{{n_i}/2}(x)+o(1))$, thus $q_{n_i}$ has Gaussian tail decay. By Cauchy-Schwarz we
    \begin{align*}
        \E\Big[ e^{\lambda |\sum_{i=1}^k{\xi_i}|}\, \Big] \leq \e \Big[ e^{2\lambda |\sum_{i=1}^k \xi_{i,1}|}\Big]^{\frac{1}{2}}  \Big[ e^{2\lambda |\sum_{i=1}^k \xi_{i,2}|}\Big]^{\frac{1}{2}}  \, .
    \end{align*}
    Also, by the inequality $e^{|x|}\leq e^x+ e^{-x}$ and independence, we obtain for $j=1,2$
    \begin{align*}
        \E\Big[ e^{2\lambda |\sum_{i=1}^k \xi_{i,j}|}\Big]^{\frac{1}{2}}\leq \Bigg( \prod_{i=1}^k \E[e^{2\lambda \xi_{i,j}}]+\prod_{i=1}^k\E[e^{-2\lambda \xi_{i,j}}]\Bigg)^{\frac{1}{2}}\leq \Big(2 e^{4c\lambda^2 n} \Big)^{\frac{1}{2}} \, ,
    \end{align*}
    therefore,
    \begin{align*}
        \E\Big[ e^{\lambda |\sum_{i=1}^k{\xi_i}|}\, \Big] \leq 2 e^{4c\lambda^2 n} \, .
    \end{align*}
    Given the inequality above and choosing $\lambda=\frac{C_w}{\sqrt{N}}$ we get that
    \begin{align*}
        \E\bigg[ e^{C_w\frac{|S^{(N)}_k|}{\sqrt{N}}}\, \Big| \,  \sft^{(N)}_i=n_i \, , 1 \leq i \leq k  \bigg] \leq 2e^{4c\, C_w^2} \, ,
    \end{align*}
    since $1\leq n \leq N$. Therefore, recalling \eqref{var_bnd} and \eqref{cond_times},we have
    \begin{align*}
        \sumtwo{z \in \Z^2,}{0 \leq n \leq N}U^{\beta_N}_N (n,z) e^{C_w\frac{|z|}{\sqrt{N}}} \leq\, 2e^{4c\, C_w^2}\,\sum_{n=0}^N U^{\beta_N}_{N}(n) \leq 2e^{4c\, C_w^2} \, \bbE\Big[({Z_{N+1}^{\beta_N}})^2\Big]\leq C \, ,
    \end{align*}
    for a constant $C=C(h,\bht,w)\in (0,\infty)$.
\end{proof}
\smallskip

\section{Proofs of Theorems \ref{onept_mom}, \ref{loc_times}, \ref{mom_est} and \ref{avg_field_thm}.} \label{main_proofs}
We are now in a position to prove the main results. We begin with Theorem \ref{mom_est}.
\vspace{0.2cm}
\begin{proof}[Proof of Theorem \ref{mom_est}]
    We first prove \eqref{av_est}. Recall from \eqref{final_expansion}
    that
    \begin{align} \label{exp_for_main_thm}
        \big|M^{\varphi,\psi}_{N,h}\big | \leq & \frac{c\norm{\psi}^h_{\infty}}{N^{h+1}}   \sum_{k\geq 1} \sum_{(I_1,\dots,I_k)  \in \mathcal{I}}  \norm{\widehat{\sfQ}^{*;I_1}_{N,0} \frac{\varphi_N^{\otimes h}}{w^{\otimes h }_N}}_{\ell^p} \prod_{i=2}^k\,  \norm{\widehat{\sfP}_{N,\bht}^{I_{i-1};I_i}}_{\ell^q \to \ell^q} \, \norm{\widehat{\sfQ}_{N,0}^{I_k;*}  w_N^{\otimes h}}_{\ell^q } \prod_{i=1}^k  \bbE\Big[|\xi|^{I_i}\Big]  \, .
    \end{align}

By Proposition \ref{Q_left_norm}, we have the following bounds on the boundary operator norms 
\begin{equation} \label{boundary_op_bound}
    \norm{\widehat{\sfQ}^{*;I_1}_{N,0} \frac{\varphi_N^{\otimes h}}{w^{\otimes h }_N}}_{\ell^p} \leq C\, q\, N^{\frac{1}{q}} \, \norm{\frac{\varphi_N}{w_N}}^h_{\ell^p} \quad \text{  and  } \quad 
    \norm{\widehat{\sfQ}_{N,0}^{I_k;*}  w_N^{\otimes h}}_{\ell^q } \leq C \, p \, N^{\frac{1}{p}}\, \norm{w_N}^h_{\ell^q} \, , 
\end{equation}
for a constant $C=C(h,w) \in (0,\infty)$.
By Propositions \ref{Q_IJ_norm} and \ref{U_norm} we also have that for all $2 \leq i \leq k$, there exists a constant $C=C(h,\bht,w) \in (0,\infty)$, such that
\begin{equation} \label{P_Ii_bound}
    \norm{\widehat{\sfP}_{N,\bht}^{I_{i-1};I_i}}_{\ell^q \to \ell^q} \leq C \, p \, q \, .
\end{equation}
By inserting the bounds \eqref{boundary_op_bound} and \eqref{P_Ii_bound} in \eqref{exp_for_main_thm} we obtain that
\begin{equation} \label{before_alogN}
    \big|M^{\varphi,\psi}_{N,h}\big | \leq \,  \frac{\norm{\psi}^h_{\infty}}{N^{h}}  \norm{\frac{\varphi_N}{w_N}}^h_{\ell^p}  \norm{w_N}^h_{\ell^q} \sum_{k \geq 1} (C\,p\,q)^k \sum_{(I_1,\dots,I_k) \in \mathcal{I}} \,\prod_{i=1}^k  \bbE\Big[|\xi|^{I_i}\Big] \, .
\end{equation}
We now distinguish two cases depending on the range of $k$.

\noindent \textbf{(Case 1)}. If $k > \lfloor \tfrac{h}{2} \rfloor$ we use the bound 
\begin{equation*}
    \prod_{i=1}^k  \bbE\big[|\xi|^{I_i}\big] \leq \big( \tfrac{C}{\log N}\big)^k \, ,
\end{equation*}
which is a consequence of the fact that $\bbE\big[|\xi|^{I_i}\big]\leq C \, \sigma^2_{N,\bht}=O(1/\log N)$, see \eqref{mixed_mom} and \eqref{xi_prp}. Therefore, in this case
\begin{equation} \label{k_large}
    \sum_{k > \lfloor \frac{h}{2} \rfloor} (C\,p\,q)^k \sum_{(I_1,\dots,I_k) \in \mathcal{I}} \,\prod_{i=1}^k  \bbE\Big[|\xi|^{I_i}\Big] \leq \sum_{k > \lfloor \frac{h}{2} \rfloor} \Big(\frac{\tilde{C}\,p\,q}{\log N}\Big)^k  \, ,
\end{equation}
for a constant $\tilde{C}=\tilde{C}(h,\bht,w) \in (0,\infty)$, which also incorporates the fact that the number of possible choices for a sequence of partitions $(I_1,\dots,I_k)$ is bounded by $C^k$ where $C=C(h)$ is some positive constant.

\noindent \textbf{(Case 2)}. The second case is when $1 \leq k \leq \lfloor \tfrac{h}{2} \rfloor$, for which we claim that there exists a constant $C=C(h,\bht) \in (0,\infty)$ such that
\begin{equation*}
    \prod_{i=1}^k  \bbE\big[|\xi|^{I_i}\big] \leq C^k\, (\log N)^{-\frac{h}{2}} \, .
\end{equation*}
To see this fix $1 \leq k \leq \lfloor \tfrac{h}{2} \rfloor$ and $(I_1,\dots,I_k) \in \mathcal{I}$, and let $I_i=\bigsqcup_{1\leq j \leq |I_i|} I_{i,j}$. By \eqref{mixed_mom} and \eqref{xi_prp}, we have that
\begin{equation*}
    \prod_{i=1}^k  \bbE\big[|\xi|^{I_i}\big] \leq C^k\, (\sigma_{N,\bht})^{\sum_{1\leq i \leq k} \sum_{1 \leq j \leq |I_i|;|I_{i,j}|\geq 2} |I_{i,j}|} \, .
\end{equation*}
From the definition of $\mathcal{I}$ (see below \eqref{MNh_expansion}), we have that 
\begin{equation*}
    \sum_{1\leq i \leq k} \sum_{1 \leq j \leq |I_i|;|I_{i,j}|\geq 2} |I_{i,j}| \geq h \, ,
\end{equation*}
since every $r \in \{1,\dots,h\}$ necessarily belongs to a non-trivial block of some partition $I_i, 1\leq i \leq k$, see the discussion below \eqref{MNh_expansion}.
Therefore, as in the derivation of \eqref{k_large}, we have that there exists a constant $\tilde{C}=\tilde{C}(h,\bht,w)\in (0,\infty)$ such that 
\begin{equation} \label{k_small}
    \sum_{1\leq k \leq \lfloor \frac{h}{2} \rfloor} (C\, p\, q)^k \sum_{(I_1,\dots,I_k) \in \mathcal{I}} \,\prod_{i=1}^k  \bbE\Big[|\xi|^{I_i}\Big] \leq (\log N)^{-\frac{h}{2}} \sum_{1\leq k \leq \lfloor \frac{h}{2} \rfloor} (\tilde{C} \, p\, q )^k \, .
\end{equation}
Combining estimates \eqref{k_large} and \eqref{k_small} we deduce from \eqref{before_alogN} that 
\begin{equation} \label{lastM_bound}
    \begin{split}
    \big|M^{\varphi,\psi}_{N,h}\big | \leq \, C\, \frac{\norm{\psi}^h_{\infty}}{N^{h}}  \norm{\frac{\varphi_N}{w_N}}^h_{\ell^p}  \norm{w_N}^h_{\ell^q}  \Bigg(\sum_{k > \lfloor \frac{h}{2} \rfloor} \Big(\frac{\tilde{C}\,p\,q}{\log N}\Big)^k 
    +(\log N)^{-\frac{h}{2}} \sum_{1\leq k \leq \lfloor \frac{h}{2} \rfloor} (\tilde{C} \, p\, q )^k \Bigg) \, .
    \end{split}
\end{equation}
Let $p,q>1$, conjugate exponents, that satisfy the growth condition
 \begin{equation} \label{growth_condition}
    \frac{\tilde{C}\,p\,q}{\log N}<\frac{1}{2} \, .
 \end{equation} 
 In particular, $p\,q\leq \sfa_* \log N$ with $\sfa_*=\sfa_*(h,\bht,w) \in (0,1)$ defined as $\sfa_*:=(2\tilde{C})^{-1}$.
 We then have that
\begin{equation} \label{Cpq_bound_1}
    \sum_{k > \lfloor \frac{h}{2} \rfloor} \Big(\frac{\tilde{C}\,p\,q}{\log N}\Big)^k \leq 2 \,\Big(\frac{\tilde{C}\,p\,q}{\log N}\Big)^{\lfloor \frac{h}{2} \rfloor+1}
\end{equation}
by summing the tail of the geometric series, which is possible due to the growth condition \eqref{growth_condition} imposed on $p,q$. On the other hand, we have that
\begin{align} \label{Cpq_bound_2}
    (\log N)^{-\frac{h}{2}} \sum_{1\leq k \leq \lfloor \frac{h}{2} \rfloor} (\tilde{C} \, p\, q )^k &\leq (\log N)^{-\frac{h}{2}} \cdot \frac{(\tilde{C} \, p \, q)^{\lfloor \frac{h}{2} \rfloor+1}-\tilde{C} \, p \, q}{\tilde{C} \, p \, q-1} \notag \\
    & \leq (\log N)^{-\frac{h}{2}} \cdot \frac{(\tilde{C} \, p \, q)^{\lfloor \frac{h}{2} \rfloor+1}}{\tilde{C} \, p \, q-1}\notag  \\
    & \leq 2(\log N)^{-\frac{h}{2}}\,  (\tilde{C} \, p \, q)^{\lfloor \frac{h}{2} \rfloor}\, ,
\end{align}
since $\tilde{C} \, p \, q-1>(\tilde{C} \, p \, q)/2$, ($pq \geq 4$ because $\tfrac{1}{p}+\tfrac{1}{q}=1$ and we can choose $\tilde{C}>1$). Combining estimates \eqref{Cpq_bound_1} and \eqref{Cpq_bound_2} we obtain that
\begin{equation*}
    \begin{split}
    \Bigg(\sum_{k > \lfloor \frac{h}{2} \rfloor} \Big(\frac{\tilde{C}\,p\,q}{\log N}\Big)^k 
    +(\log N)^{-\frac{h}{2}} \sum_{1\leq k \leq \lfloor \frac{h}{2} \rfloor} (\tilde{C} \, p\, q )^k \Bigg) &\leq 2 \,\Big(\frac{\tilde{C}\,p\,q}{\log N}\Big)^{\lfloor \frac{h}{2} \rfloor+1} +2(\log N)^{-\frac{h}{2}}\,  (\tilde{C} \, p \, q)^{\lfloor \frac{h}{2} \rfloor} \\
    & \leq 4 \,\Big(\frac{\tilde{C}\,p\,q}{\log N}\Big)^{\frac{h}{2}} \, ,
    \end{split}
\end{equation*}
by using that $\tfrac{\tilde{C}\,p\,q}{\log N} \leq \tfrac{1}{2}$ and $\lfloor \tfrac{h}{2} \rfloor \leq \tfrac{h}{2}<\lfloor \tfrac{h}{2} \rfloor+1$. Inserting this bound to \eqref{lastM_bound} we finally obtain that 
\begin{equation}
\big|M^{\varphi,\psi}_{N,h}\big | \leq \, \Big(\frac{C\,p\,q}{\log N}\Big)^{\frac{h}{2}}\, \frac{1}{N^{h}}  \norm{\frac{\varphi_N}{w_N}}^h_{\ell^p} \norm{\psi}^h_{\infty} \norm{w_N}^h_{\ell^q} \, ,
\end{equation}
for a constant $C=C(h,\bht,w) >\tilde{C}$, which establishes \eqref{av_est}.

    Let us now prove \eqref{onept_est}. By choosing $\varphi:=\delta_0^{\ms (N)}:=N\, \ind_{\{x=0\}}$, $\psi \equiv 1 $ and $w(x)=e^{-|x|}$, we deduce from \eqref{av_est} that
    \begin{align}  \label{final_onept_est}
        \Big|\bbE\big[(\widebar{Z}_{N}^{\beta_N}
            )^h\big]\Big| \leq  \Big(\frac{C\,p\,q}{\log N}\Big)^{\frac{h}{2}} \norm{w_N}^{h}_{\ell^q} = \Big(\frac{C\,p\,q}{\log N}\Big)^{\frac{h}{2}}\cdot N^{\frac{h}{q}}\cdot  \frac{1}{N^{\frac{h}{q}}} \norm{w_N}^{h}_{\ell^q} \, .
    \end{align}
    Since $w(x)=e^{-|x|}$ is decreasing in the radial direction we have
    \begin{align} \label{w_N_bnd}
        \frac{1}{N^{\frac{h}{q}}} \norm{w_N}^{h}_{\ell^q} \leq \bigg( \frac{1}{N} +\frac{1}{N} \int_{\R^2}  \,  e^{-q \frac{|x|}{\sqrt{N}}}\, \dd x\bigg)^{\frac{h}{q}} =  \bigg( \frac{1}{N} + \int_{\R^2} \,  e^{-q |x|} \, \dd x \bigg)^{\frac{h}{q}} =\bigg( \frac{1}{N} + \frac{2 \pi}{q^2}  \bigg)^{\frac{h}{q}}\leq e^{(2 \pi h)/q^3} \, .
    \end{align}
    We choose $q=q_N:= \sfa \log N$ with $\sfa=\sfa(h,\bht,w) \in (0,1)$ small enough such that $\tfrac{C\,p\,q}{\log N}<\tfrac{1}{2}$ (and therefore \eqref{growth_condition} is satisfied). For this choice of $q$ we have by \eqref{w_N_bnd} that
    \begin{align} \label{w_N_lim}
         \frac{1}{N^{\frac{h}{q}}} \norm{w_N}^{h}_{\ell^{q}} \leq \, e^{O\big((\log N)^{-3}\big)}\leq C \, .
    \end{align}
    Furthermore, again with $q=q_N = \sfa \log N$ and thus $p=p_N= 1+o(1)$, since $\frac{1}{p}+\frac{1}{q}=1$, we get
    \begin{align} \label{a_exp_a}
        \Big(\frac{C\,p\,q}{\log N}\Big)^{\frac{h}{2}} \cdot N^{\frac{h}{q}}\leq 2^{-\frac{h}{2}} \, \exp\big(\tfrac{h}{\sfa}\big) <\infty \, ,
    \end{align}
    since $\tfrac{C\,p\,q}{\log N}<\tfrac{1}{2}$. We note that the parameter $\sfa=\sfa(h,\bht,w)$ on the right-hand side of 
    \eqref{a_exp_a} depends non-trivially on $h$, and therefore the order of the bound in \eqref{a_exp_a} is not just exponential in $h$.
    However, we also note that the dependence on $h$ deduced from \eqref{a_exp_a} does not capture the true growth of the
    moments, which is as in \eqref{limiting_moments}. 
    Finally, by \eqref{final_onept_est}, \eqref{w_N_lim} and \eqref{a_exp_a}, we obtain that
    \begin{align*}
        \sup_{N \in \N} \bbE\big[(\widebar{Z}^{\beta_N}_{N})^h\big] < \infty \, .
    \end{align*}
\end{proof}

\begin{proof}[Proof of Theorem \ref{onept_mom}]
    By binomial expansion, for $h \in \N$ we have that
    \begin{align*}
        \bbE\big[(Z^{\beta_N}_{N
                })^h\big]= \sum_{k=0}^h \binom{h}{k}\bbE \big[ (\widebar{Z}^{\beta_N}_{N
            })^k   \big]\leq \sum_{k=0}^h \binom{h}{k} \Big|\bbE\big[ (\widebar{Z}^{\beta_N}_{N
            })^k   \big] \Big | \, .
    \end{align*}
    Therefore, by estimate \eqref{onept_est} of Theorem \ref{mom_est}, for every $h \geq 3$ we obtain that
    $\sup_{N \in \N} \bbE\big[(Z^{\beta_N}_{N
                })^h\big] < \infty$. Hence, for every $h \geq 0$ the sequence $\Big\{(Z^{\beta_N}_{N
            })^h\Big\}_{N \geq 1}$ is uniformly integrable and therefore, by Theorem \ref{TheoremA} for every $h \geq 0$,
    \begin{align*}
        \lim_{N \to \infty}  \bbE\big[(Z^{\beta_N}_{N
                })^h\big] = \bbE \big[ \exp\big(\rho_{\bht}\, h\, \sfX -\tfrac{1}{2} \rho^2_{\bht}\, h \big) \big]= \exp \Big( \tfrac{h(h-1)}{2}\, \rho^2_{\bht}\Big)=\bigg( \frac{1}{1-\bht^2}\bigg)^{\frac{h(h-1)}{2}}\, .
    \end{align*}
    As can be seen in \cite{CSZ20}, section 3, \eqref{conc_pty} implies that for all $h>0$,
    \begin{align*}
        \sup_{N \in \N} \bbE\big[(Z^{\beta_N}_{N})^{-h}\big ] < \infty,
    \end{align*}
    which in combination with Theorem \ref{TheoremA} implies the convergence of negative moments.
\end{proof}
\smallskip

\begin{proof}[Proof of Theorem \ref{loc_times}]
    We note that if we choose the law of the environment $\omega$ to be Gaussian, i.e. $ \omega \sim \cN(0,1)$, then for $h \in \N$
    \begin{align*}
        \bbE\big[(Z^{\beta_N}_{N
                })^h\big ]= \E^{\otimes h} \bigg[\exp \Big(\beta_N^2 \sum_{1 \leq i<j \leq h} \sfL^{(i,j)}_N \Big)\bigg] = \E^{\otimes h} \bigg[\exp \Big(\frac{ \bht^2 \, \pi}{\log N}\big(1+o(1)\big) \sum_{1 \leq i<j \leq h} \sfL^{(i,j)}_N \Big)\bigg] \, .
    \end{align*}
    Therefore, by Theorem \ref{onept_mom} we have that
    \begin{align} \label{mgf_conv}
        \E^{\otimes h} \bigg[\exp \Big(\frac{ \bht^2 \, \pi}{\log N} \sum_{1 \leq i<j \leq h} \sfL^{(i,j)}_N \Big)\bigg] \xrightarrow{N \to \infty} \bigg( \frac{1}{1-\bht^2}\bigg)^{\frac{h(h-1)}{2}} \, ,
    \end{align}
    for all $\bht \in [0,1)$. The right-hand side of \eqref{mgf_conv} is equal to $M_Y(\bht^2)$, where $M_Y(t):=\E[e^{tY}]$ denotes the moment generating function of a random variable $Y$ with law $\Gamma\big( \frac{h(h-1)}{2},1\big)$. By exercise 9, chapter 4 in \cite{K97}, \eqref{mgf_conv} implies the convergence of $\frac{\pi}{\log N} \sum_{1 \leq i<j \leq h} \sfL^{(i,j)}_N$ in law, to a $\Gamma\big( \frac{h(h-1)}{2},1\big)$ distribution. 
\end{proof}
\smallskip

\begin{proof}[Proof of Theorem \ref{avg_field_thm}]
    We are going to show that for all $h \in \N$ with $h \geq 3$ we have that
    \begin{align} \label{avg_stmt}
        \sup_{N \in \N}\,	(\log N)^{\frac{h}{2}}\, |M^{\varphi,\psi}_{N,h}| < \infty \, .
    \end{align}
    In that case we obtain uniform integrability of $(\log N)^{\frac{h}{2}}\cdot \big(\widebar{Z}^{\beta_N}_{N}(\varphi,\psi)\big)^h$
    for all $h \in \N$  and the convergence of moments in Theorem \ref{avg_field_thm} follows by Theorem \ref{TheoremB}. But, \eqref{avg_stmt} is an immediate consequence of \eqref{av_est} of Theorem \ref{mom_est}. 
    Indeed, let us fix $p,q \in (1,\infty)$ such that $\frac{1}{p}+\frac{1}{q}=1$. By \eqref{av_est} of Theorem \ref{mom_est} we have that 
    \begin{align} \label{avg_Cpq}
       (\log N)^{\frac{h}{2}} \big|M^{\varphi,\psi}_{N,h}\big|\leq \, (C\,p\,q)^{\frac{h}{2}}\frac{1}{N^h}\norm{\frac{\varphi_N}{w_N}}^h_{\ell^p}\norm{w_N}^{h}_{\ell^q}  \norm{\psi_N}^h_{\infty}  .
    \end{align}
    Furthermore, by Riemann approximation we have that
    \begin{align} \label{riem_approx}
        \frac{1}{N^h}\norm{\frac{\varphi_N}{w_N}}^h_{\ell^p}\norm{w_N}^{h}_{\ell^q}  \norm{\psi_N}^h_{\infty} =\frac{1}{N^{\frac{h}{p}}}\norm{\frac{\varphi_N}{w_N}}^h_{\ell^p}\frac{1}{N^{\frac{h}{q}}}\norm{w_N}^{h}_{\ell^q}  \norm{\psi_N}^h_{\infty} \leq \,C\, \norm{\frac{\varphi}{w}}^h_{L^p}\norm{w}^{h}_{L^q}  \norm{\psi}^h_{\infty} \, .
    \end{align}
    Therefore, by \eqref{avg_Cpq} and \eqref{riem_approx} we obtain that
    \begin{align*}
        \sup_{N \in \N}\,	(\log N)^{\frac{h}{2}}\, |M^{\varphi,\psi}_{N,h}| < \infty \, ,
    \end{align*}
    which concludes the proof.
\end{proof}

\vspace{0.5cm}
\appendix
\section{Some technical estimates} \label{technical_estimates}
We state here the integral estimates we used for proving Propositions \ref{Q_IJ_norm} and \ref{Q_left_norm}.
\begin{lemma} \label{sums_bnds}
    Let $\lambda \geq 1$, $p>1$, $a<\tfrac{1}{p}$. Then,
    \vspace{0.2cm}
    \begin{align}
         & \, \, \, \sum_{y \in \Z^2 } \frac{1}{\big(\lambda+|y|^2\big)^r} \leq  \frac{c}{\lambda^{r-1}}
         &                                                                                                                                                & \quad  \text{ if }  r \geq 2  \, ,      \label{s1} \\
         & \, \, \, \sum_{y \in \Z^2} \frac{1}{\big(\lambda+|y|^2\big)^r\big(1+ |y|^{2a}\big)^p }\leq \,  \frac{c}{a p\,(1-ap)\lambda^{r-1+ap}}\label{s2}
         &                                                                                                                                                & \quad\text{ if } r\geq 1 \, ,
    \end{align}
    for a constant $c\in (0,\infty)$, that does not depend on $\lambda, p, a$ or $r$.
\end{lemma}
\begin{proof}
    We note that since $\displaystyle y \mapsto  \frac{1}{\big(\lambda+|y|^2\big)^r}$ and $\displaystyle y \mapsto\frac{1}{\big(\lambda+|y|^2\big)^r \big(1+ |y|^{2a}\big)^p } $ are decreasing in the radial direction we have that
    \begin{align} \label{reduc1}
        \sum_{y \in \Z^2}\frac{1}{\big(\lambda+|y|^2\big)^r} \leq  \frac{1}{\lambda^r} + \int_{\R^2} \, \frac{1}{\big(\lambda+|y|^2\big)^r} \, \dd y \,
    \end{align}
    and
    \begin{align} \label{reduc2}
        \sum_{y \in \Z^2} \frac{1}{\big(\lambda+|y|^2\big)^r  \big(1+ |y|^{2a}\big)^p }  \leq \frac{1}{\lambda^r}+ \int_{\R^2} \, \frac{1}{\big(\lambda+|y|^2\big)^r |y|^{2ap}}\,\dd y \, .
    \end{align}
    In order to prove \eqref{s1}, we switch to polar coordinates in \eqref{reduc1}, so that
    \begin{align} \label{first_comp}
        \int_{\R^2} \, \frac{1}{\big(\lambda+|y|^2\big)^r}\,  \dd y = 2 \pi \int_{0}^{\infty} \frac{\rho}{(\lambda+\rho^2)^r}\, \dd \rho\, =\pi \cdot \frac{(\lambda+\rho^2)^{1-r}}{1-r} \bigg|_{\rho=0}^{\rho=\infty}  =\frac{\pi}{r-1} \, \frac{1}{\lambda^{r-1}} \, .
    \end{align}
    Therefore, by \eqref{reduc1} and \eqref{first_comp} we get that
    \begin{align*}
        \sum_{y \in \Z^2}\frac{1}{\big(\lambda+|y|^2\big)^r} \leq  \frac{1}{\lambda^r} + \frac{\pi}{r-1} \frac{1}{\lambda^{r-1}} = \frac{1}{\lambda^{r-1}} \bigg( \frac{1}{\lambda}+\frac{\pi}{r-1 }\bigg) \, .
    \end{align*}
    Thus, since $r\geq 2$ and $\lambda\geq 1$ we conclude \eqref{s1}
    with $c=\pi+1$.

    For \eqref{s2} we split the integral in \eqref{reduc2} into two regions,
    \begin{align*}
        \int_{\R^2} \frac{1}{\big(\lambda+|y|^2\big)^r |y|^{2ap}}\, \dd y \, = \underbrace{\int_{|y|\leq\sqrt{\lambda}}  \frac{1}{\big(\lambda+|y|^2\big)^r |y|^{2ap}}\, \dd y \,}_{:=I_1}+ \underbrace{\int_{|y| > \sqrt{\lambda}} \, \frac{1}{\big(\lambda+|y|^2\big)^r |y|^{2ap}} \, \dd y \,}_{:=I_2} \, .
    \end{align*}
    First,
    \begin{align*}
        I_1 \leq \frac{1}{\lambda^r}\int_{|y| \leq \sqrt{\lambda}} \, \frac{1}{|y|^{2ap}}\, \dd y = \frac{2 \pi}{\lambda^r} \int_{0}^{\sqrt{\lambda}} \, \frac{1}{\rho^{2ap -1}}\, \dd \rho = \frac{\pi}{\lambda^{r}} \frac{\lambda^{1-ap}}{1-ap}=\frac{\pi}{1-ap} \frac{1}{\lambda^{r-1+ap}} \, .
    \end{align*}
    Similarly,
    \begin{align*}
        I_2 \leq  \int_{|y|>\sqrt{\lambda}} \, \frac{1}{|y|^{2r+2ap}}  \dd y =2 \pi \int_{\sqrt{\lambda}}^{\infty}   \frac{1}{\rho^{2r+2ap-1}} \, \dd \rho\,= &
        \frac{\pi}{r-1+ap}  \frac{-1}{\rho^{2r+2ap-2}}\bigg |^{\rho=\infty}_{\rho=\sqrt{\lambda}}  \notag                                                                                                          \\
        =                                                                                                                                                     & \frac{\pi}{r-1+ap} \frac{1}{\lambda^{r-1+ap}} \, .
    \end{align*}
    Therefore,
    \begin{align} \label{I_1+I_2}
        \int_{\R^2}  \, \frac{1}{\big(\lambda+|y|^2\big)^r |y|^{2ap}}\, \dd y \, = I_1+I_2\leq & \, \frac{\pi}{1-ap} \frac{1}{\lambda^{r-1+ap}} +\frac{\pi}{r-1+ap} \frac{1}{\lambda^{r-1+ap}} \notag \\
        =                                                                                   & \, \frac{\pi \, r}{(1-ap)(r-1+ap)} \frac{1}{\lambda^{r-1+ap}} \, .
    \end{align}
    By \eqref{reduc2} and \eqref{I_1+I_2} we thus obtain
    \begin{align} \label{ap_1}
        \sum_{y \in \Z^2} \frac{1}{\big(\lambda+|y|^2\big)^r  \big(1+ |y|^{2a}\big)^p }  \leq & \frac{1}{\lambda^r}+ \int_{\R^2}  \, \frac{1}{\big(\lambda+|y|^2\big)^r |y|^{2ap}}\, \dd y  \notag \\
        \leq                                                                                  & \frac{1}{\lambda^r} + \frac{\pi \, r}{(1-ap)(r-1+ap)} \frac{1}{\lambda^{r-1+ap}} \, .
    \end{align}
    Note that
    \begin{align*}
        \frac{\pi \, r}{(1-ap)(r-1+ap)} \leq \frac{\pi }{(1-ap)ap} \, ,
    \end{align*}
    since that inequality is equivalent to $(r-1)(ap-1)\leq 0$, which is valid since we have assumed that $a\,p<1$ and $r \geq 2$. Therefore,
    \begin{align} \label{ap_2}
        \frac{1}{\lambda^r} + \frac{\pi \, r}{(1-ap)(r-1+ap)} \frac{1}{\lambda^{r-1+ap}} \leq  \,\frac{1}{\lambda^r} + \frac{\pi }{(1-ap)a p} \frac{1}{\lambda^{r-1+ap}}
        =    & \frac{1}{\lambda^{r-1+ap}} \bigg( \frac{1}{\lambda^{1-ap}}+\frac{\pi}{(1-ap)ap} \bigg)  \notag \\
        \leq & \frac{1}{\lambda^{r-1+ap}} \bigg(  1+\frac{\pi}{(1-ap)ap}  \bigg) \, ,
    \end{align}
    since $\lambda \geq 1$ and $1-ap>0$, by assumption.
    Last, we have that
    \begin{align} \label{ap_3}
        1+\frac{\pi}{(1-ap)ap} = \frac{\pi+ap(1-ap)}{ap(1-ap)} \leq \frac{1+\pi}{ap(1-ap)} \, .
    \end{align}
    Hence, by \eqref{ap_1}, \eqref{ap_2} and \eqref{ap_3},
    \begin{align*}
        \sum_{y \in \Z^2} \frac{1}{\big(\lambda+|y|^2\big)^r  \big(1+ |y|^{2a}\big)^p }  \leq \frac{c}{(1-ap)\, ap \,\lambda^{r-1+ap}} \, ,
    \end{align*}
    with $c=1+\pi$, thus concluding the proof of \eqref{s2}.
\end{proof}

\begin{lemma} \label{sums_bnds_2}
    There exists a constant $c \in (0,\infty)$ such that uniformly in $A,\lambda,p \geq 1$, 
    \begin{align}
        \sumtwo{y \in \Z^2}{  |y|\leq \sqrt{A}} \frac{1}{\lambda+|y|^2}\leq c \log\Big( 1+\frac{A}{\lambda}\Big) \, \label{s3} \, ,
    \end{align}
    \begin{align} \label{log^p}
        \bigg(\int_{1}^A   \big(\log\big(\tfrac{A}{x}\big)\big)^{p}  \,\dd x\,\bigg)^{\frac{1}{p}} \leq p \, A^{\frac{1}{p}}\, ,
        \,
    \end{align}
    and
    \begin{align}
         & \sumtwo{y \in \Z^2}{|y| \leq 2\sqrt{A}} \Big(\log\Big(1+ \frac{A}{1+|y|^2}\Big)\Big)^p \leq c \, A \, p^p \, . \label{s4}
    \end{align}
    \vspace{0.2cm}
\end{lemma}

\begin{proof}
    For \eqref{s3}, using the same reasoning as in the proof of Lemma \ref{sums_bnds} we have
    \begin{align}
        \sumtwo{y \in \Z^2}{  |y|\leq \sqrt{A}} \frac{1}{\lambda+|y|^2}\leq \frac{1}{\lambda}+ \int_{|y|\leq \sqrt{A}}  \,  \frac{1}{\lambda+|y|^2}\, \dd y  \label{reduc3}\, .
    \end{align}
    Switching to polar coordinates in \eqref{reduc3} we have
    \begin{align} \label{h=1}
        \int_{|y|\leq \sqrt{A}}  \frac{1}{\lambda+|y|^{2}} \,  \dd y \,= 2 \pi \int_{0}^{\sqrt{A}}  \frac{\rho}{\lambda+\rho^2} \,\dd \rho = \pi \,\log(\lambda+\rho^2)\Big|_{\rho=0}^{\rho=\sqrt{A}}=\pi \log\Big(1+ \frac{A}{\lambda}\Big) \, .
    \end{align}
    A simple computation shows that when $\lambda \geq 1$, one has that $\frac{1}{\lambda}\leq 2 \log\Big(1+\frac{1}{\lambda}\Big)\leq 2 \log\Big(1+\frac{A}{\lambda}\Big) $, the latter following since $A \geq 1$, by assumption. Therefore,
    by \eqref{reduc3} and \eqref{h=1} we have that
    \begin{align*}
        \sumtwo{y \in \Z^2}{  |y|\leq \sqrt{A}} \frac{1}{\lambda+|y|^2}\leq \frac{1}{\lambda}+ \int_{|y|\leq \sqrt{A}}  \,  \frac{1}{\lambda+|y|^2}\, \dd y\leq (2+\pi) \log\Big(1+\frac{A}{\lambda}\Big) \, ,
    \end{align*}
    which implies \eqref{s3} with $c=2+\pi$.

    Let us now prove \eqref{log^p} and \eqref{s4}.
    First, we prove \eqref{log^p}.
    We have
    \begin{align} \label{log_int}
        \frac{1}{A}\int_{1}^A   \big(\log\big(\tfrac{A}{x}\big)\big)^{p}\, \dd x\,   \stackrel{y=A\, e^{-u}}{=\joinrel=}
        \int_{0}^{\log A}  e^{-u} \, u^p \, \dd u  \leq \,  \Gamma(p+1) \leq p^p\, ,
    \end{align}
    since $\Gamma(p+1)\leq p^p$ for $p\geq 1$. After raising both sides of \eqref{log_int} to the $\frac{1}{p}$ we get \eqref{log^p}.

    To prove \eqref{s4} we first note that
    \begin{align}
        \sumtwo{y \in \Z^2}{|y| \leq 2\sqrt{A}} \Big(\log\Big(1+ \frac{A}{1+|y|^2}\Big)\Big)^p \leq(\log (1+A))^p + \int_{|y|\leq 2 \sqrt{A} }  \, \Big(\log\Big(1+\frac{A}{1+|y|^2}\Big)\Big)^p \,\dd y\label{reduc4} \, .
    \end{align}
    Using polar coordinates in \eqref{reduc4} we compute
    \begin{align*}
        \int_{|y|\leq 2 \sqrt{A}}  \, \Big(\log\Big(1+\frac{A}{1+|y|^2}\Big)\Big)^p\, & \dd y= 2\pi \int_{0}^{2 \sqrt{A} }  \, \rho \Big(\log\Big(1+\frac{A}{1+\rho^2}\Big)\Big)^p \, \dd \rho\,         \\
                                                                                            & \stackrel{u=1+\rho^2}{=\joinrel=}\,  \pi \int_{1}^{1+4A}  \Big(\log\Big( 1+ \frac{A}{u}\Big)\Big)^p\, \dd u \, .
    \end{align*}
    Furthermore,
        \begin{align*}
            \pi \int_{1}^{1+4A}  \Big(\log\Big( 1+ \frac{A}{u}\Big)\Big)^p \, \dd u \leq \pi \int_{1}^{1+4A} \Big(\log\Big(  \frac{1+5A}{u}\Big)\Big)^p\,  \dd u \leq\, \pi \int_{1}^{1+5A} \Big(\log\Big(  \frac{1+5A}{u}\Big)\Big)^p \dd u \, .
        \end{align*}
    Note that by \eqref{log^p}, we further have that
    \begin{align*}
        \pi \int_{1}^{1+5A} \Big(\log\Big(  \frac{1+5A}{u}\Big)\Big)^p \dd u \leq  (1+5A) \, \pi\,  p^p \, \leq 6\,A\, \pi \, p^p \, ,
    \end{align*}
    since $A \geq 1$. Combining this inequality with \eqref{reduc4} we get that
    \begin{align} \label{last_log}
        \sumtwo{y \in \Z^2}{|y| \leq 2\sqrt{A}} \Big(\log\Big(1+ \frac{A}{1+|y|^2}\Big)\Big)^p \leq \log (1+A)^p + 6 \,A\, \pi\, p^p \, .
    \end{align}
    We are going to prove that for all $A \geq 1$,
    \begin{align*}
        (\log(1+A))^p \leq \frac{1}{\sqrt{e}-1}\, A \, p^p \, ,
    \end{align*}
    thus deducing inequality \eqref{s4}, via \eqref{last_log}, with $c=\frac{1}{\sqrt{e}-1}+6 \pi$. To this end, consider $k_p(x):=\frac{(\log(1+x))^p}{x}$ for $x \geq 0$ and $p\geq 1$. We have that
    \begin{align*}
            k'_p(x):=\frac{(\log(1+x))^{p-1}}{x}\Big(\frac{p}{1+x}-\frac{\log(1+x)}{x}\Big) \, ,
        \end{align*}
    therefore, $k_p$ is increasing in $[0,x_p]$ and decreasing in $[x_p,\infty)$, where $x_p \geq 0$ is the solution to the equation $k'_p(x_p)=0$, or equivalently
    \begin{align} \label{x_p_eq}
        p=\frac{(1+x_p)\,\log(1+x_p)}{x_p} \, .
    \end{align}
    By working with $g(x):=\frac{(1+x)\log(1+x)}{x}$, one can see that equation \eqref{x_p_eq} has a unique solution $x_p \geq 0$ for every $p\geq 1$, since $g'(x)>0$ for all $x > 0$, $\lim_{x \downarrow 0} g(x)=1$ and $\lim_{x \to \infty} g(x)=\infty$.
    We distinguish two cases:

    Suppose first that $x_p \geq 1$.
    Then
    \begin{align} \label{x_p_pos}
        \log(1+x_p) \leq p  \leq 2\,\log(1+x_p) \, ,
    \end{align}
    by \eqref{x_p_eq} and since $x_p \geq 1$. Therefore, in this case, for all $x \geq 1$,
    \begin{align*}
            \frac{(\log(1+x))^p}{x}=k_p(x)\leq k_p(x_p)= \frac{(\log(1+x_p))^p}{x_p}\leq \frac{p^p}{x_p}\leq \frac{p^p}{e^{\frac{p}{2}}-1} \, ,
        \end{align*}where the last two inequalities follow by the first and second inequality in \eqref{x_p_pos}, respectively. Since, $p\geq 1$ we have that $e^{\frac{p}{2}}-1\geq \sqrt{e}-1$, thus we conclude that in the case where $x_p \geq 1$ we have for all $x \geq 1$,
    \begin{align} \label{k_p_bnd_1}
            k_p(x)=\frac{(\log(1+x))^p}{x}\leq \frac{1}{\sqrt{e}-1}\, p^p \, .
        \end{align}
    Moving to the second case, i.e. $0 \leq x_p < 1$, we have that since $k_p$ is decreasing in $[1,\infty)\subset [x_p,\infty)$, we have that for all $x \geq 1$,
    \begin{align} \label{k_p_bnd_2}
            k_p(x)=\frac{(\log(1+x))^p}{x}\leq k_p(1)= (\log(2))^p < 1 \, ,
        \end{align}since $p\geq 1$ and $\log 2<1$. Therefore, by \eqref{k_p_bnd_1}, \eqref{k_p_bnd_2} and since $p \geq 1$, for all $x \geq 1$ and $p\geq 1$ we have that
    \begin{align} \label{k_p_last_bound}
        k_p(x) \leq \max\Big\{1,\frac{1}{\sqrt{e}-1}\Big\} \, p^p =\frac{1}{\sqrt{e}-1} \, p^p\, .
    \end{align}
    Recalling that $k_p(x)=\frac{(\log(1+x))^p}{x}$ and applying \eqref{k_p_last_bound} to \eqref{last_log} for $x=A \geq 1$ we get that
    \begin{align*}
            \sumtwo{y \in \Z^2}{|y| \leq 2\sqrt{A}} \Big(\log\Big(1+ \frac{A}{1+|y|^2}\Big)\Big)^p \leq (\log (1+A))^p + 6 \,A\, \pi\, p^p \leq \frac{1}{\sqrt{e}-1}\, A \, p^p + 6 \,A \, \pi \, p^p =c \, A \, p^p \, ,
        \end{align*}with $c=\frac{1}{\sqrt{e}-1}+ 6 \pi$, thus concluding the proof of \eqref{s4}.
\end{proof}

\vskip 0.5cm

\noindent \textbf{Acknowledgements.}
We would like to thank Francesco Caravenna and Rongfeng Sun for many useful discussions. We would also like to thank Cl\'ement Cosco
for pointing to the problem of the convergence of the exponential of the rescaled log-partition function to GMC, which motivated this
work, and for making us aware of \cite{CZ21} prior to publication.
We would also like to thank Mo Dick Wong and Perla Sousi for useful discussions.
D. L. acknowledges financial support from EPRSC through grant EP/HO23364/1 as part of the MASDOC DTC at the University of Warwick.
N.Z. work was supported by EPSRC through grant EP/R024456/1.

\end{document}